\documentclass[11pt]{article}
\usepackage[margin=1in]{geometry}
\usepackage{amsmath,amssymb,amsthm}
\usepackage{natbib}
\usepackage[hidelinks]{hyperref}
\hypersetup{pdftitle={Seed thresholds and degree variance in heterogeneous bootstrap percolation with growing degrees},pdfauthor={Aleksandr Rodionov},pdfsubject={Bootstrap percolation on random graphs},pdfkeywords={bootstrap percolation, random graphs, seed thresholds, degree variance}}
\newtheorem{theorem}{Theorem}
\newtheorem{lemma}[theorem]{Lemma}
\newtheorem{corollary}[theorem]{Corollary}
\newtheorem{proposition}[theorem]{Proposition}
\newcommand{\E}{\mathbb E}

\title{Seed thresholds and degree variance in heterogeneous bootstrap percolation with growing degrees}
\author{Aleksandr Rodionov\\
\small Yandex, Saint Petersburg, Russia\\
\small ORCID: \href{https://orcid.org/0009-0006-6710-923X}{0009-0006-6710-923X}\\
\small \href{mailto:rexarrior@yandex.ru}{\texttt{rexarrior@yandex.ru}}}
\date{Preprint --- revised 23 November 2026}
\begin{document}
\maketitle
\begin{abstract}
We study bootstrap percolation with independent vertex thresholds taking
values one and two, with threshold-one probability $(1-c/d)/d$ for fixed
$c>0$. The seed set is chosen uniformly among sets of a prescribed
deterministic size, independently of the graph and thresholds.
We prove threshold statements at fixed relative margins. For uniform simple
graphs with prescribed nonnegative integer degrees of even sum and exact
mean $d$, assume $\max_i|d_i-d|\le C\sqrt d$ and
$c+1-v_n\ge\kappa>0$, where $v_n=\operatorname{Var}(D)/d$ and
$C,\kappa$ are fixed. When $d\to\infty$ and $d=o(n^{1/7})$, the
leading seed scale is $n(c+1-v_n)^2/(2d^4)$, without requiring a limit
of $v_n$. At fixed relative margins below and above this scale, the
final active set has size $O(n/d^3)$ and $n-o(n)$, respectively,
with high probability. For regular graphs, the supercritical conclusion
strengthens to full activation. For $G(n,d/n)$, the same $O(n/d^3)$
and $n-o(n)$ conclusions hold when $d\to\infty$ and $d^5/n\to0$,
with seed scale $nc^2/(2d^4)$.
Thus regular and independent-edge graphs have different coefficients
at the same asymptotic mean degree. Local exploration estimates yield
explicit inactive remainders and survive conditioning on simplicity
in the prescribed model. We also quantify the precision obstruction
to static inclusion transfer and compute deterministic response-barrier
corrections, without identifying a shrinking random critical window.

A separate appendix gives the corresponding fixed-margin terminal bounds
for iid thresholds of fixed bounded support, with unchanged threshold-one
probability and threshold-two mass bounded away from zero.

\end{abstract}

\section{Model and main results}
On a finite simple graph on $[n]=\{1,\ldots,n\}$, start from a seed
set $A_0$ and repeatedly activate each inactive vertex having at least
$R_i$ active neighbours. Activation is permanent; write $A_\infty$
for the final active set. Throughout, $n\to\infty$ and logarithms
are natural. For a fixed $c>0$, the thresholds are mutually independent
and independent of the graph, with
\begin{equation}\label{eq:types}
 \Pr(R_i=1)=\pi=\frac{1-c/d}{d},\qquad
 \Pr(R_i=2)=1-\pi.
\end{equation}
The law~\eqref{eq:types} governs the main text and
Appendices~\ref{lup:appendix}--\ref{app:rossi-overlap};
Appendix~\ref{app:bounded-tail} states its separate bounded-tail law.
The seed set is independent of the graph and all thresholds and is
uniform among subsets of a prescribed deterministic integer size
$a=a_n\in\{0,\ldots,n\}$. Write $\alpha=a/n$.
All probability statements average over the graph, thresholds and seeds.
With high probability means with probability tending to one along each
sequence satisfying the stated hypotheses. A bound $X=O(f_n)$ with high
probability means that some fixed finite $B$ satisfies
$\Pr(X\le Bf_n)\to1$.

\subsection{Excess degree and the seed scale}
Degree deviations of order $\sqrt d$ vanish relative to the mean,
yet can change the leading seed threshold in this mixture.
The threshold-one response is within order $d^{-1}$ of criticality,
the same order as the normalized excess-degree correction.
For prescribed degrees let $D$ have the empirical degree distribution,
let $d=\mathbb E D$, and put $v_n=\operatorname{Var}(D)/d$.
We assume a fixed positive gap $c+1-v_n\ge\kappa>0$.
The excess-degree identity and its deficit are
\[
 \frac{\mathbb E[D(D-1)]}{d}=d-1+v_n,\qquad
 b=\frac{c+1-v_n}{d}+\frac{c(v_n-1)}{d^2}.
\]
Thus $b\asymp d^{-1}$ uniformly under the fixed parameter bounds.
If additionally $v_n\to v$, the leading coefficient specializes to
$(c+1-v)^2/2$. The main theorem does not require this convergence.

For regular degrees the remaining degree is $d-1$: the edge used to
reach the vertex is no longer available. For a Poisson degree law of
mean $d$, size bias followed by removal of the arrival edge leaves
mean $d$. This one-edge difference changes the deficit from
$(c+1)/d+O(d^{-2})$ to $c/d$.
The next response coefficient is $K\sim d^2$, producing the seed scale
$b^2/(2K)$ and hence the two different leading coefficients.
The theorems below justify this deterministic prediction.

For prescribed degrees we quantitatively refine the established exponential
exploration at a bottleneck of order $n/d^3$, obtaining failure probabilities
that remain negligible after the $\exp\{O(d^2)\}$ cost of conditioning
on simplicity. For independent edges we
specialize the established independent-mark and maximal empirical-process
arguments to the mixed response, obtaining the range $d^5/n\to0$.
Both arguments give terminal bounds at fixed relative margins.

The resulting regular/independent-edge separation rigorously realizes the prediction
of the established response equations. In the regular model, a spectral
estimate of \citet{cook2018} upgrades the
supercritical conclusion to full activation throughout the same degree range.
The slow-degree supercritical
overlap with finite-time approximation is recorded in
Appendix~\ref{app:rossi-overlap}. Higher-moment refinements concern
deterministic barriers, rather than the centre of a random critical window.

Appendix~\ref{app:bounded-tail} separately extends the fixed-margin
terminal conclusions to iid thresholds supported on $\{1,\ldots,r_0\}$
for fixed $r_0$, with unchanged threshold-one probability and
$q_{2,n}\ge q_*>0$. Its leading seed
scales are the two-threshold scales divided by $q_{2,n}$.

\subsection{Threshold statements}
Define the residual bounds
\begin{equation}\label{eq:residuals}
 \mathcal R_{\rm P}(n,d,C)
   =n\bigl[d\,2^{-d+C\sqrt d}+n^{-1/3}\bigr],\qquad
 \mathcal R_{\rm E}(n,d)
   =n\bigl[(1+d)e^{-d/2+d/n}+n^{-1/3}\bigr].
\end{equation}

\begin{theorem}[P: prescribed degrees]\label{thm:v}
Fix $c>0$, $C\ge0$, $\kappa>0$ and $\gamma\in(0,1)$.
Let $d_1,\ldots,d_n$ be nonnegative integers with even sum $nd$,
where $d=n^{-1}\sum_i d_i$ is the exact mean. Let $D$ have the
empirical degree distribution, and assume
\begin{equation}\label{eq:degrees}
 \max_i|d_i-d|\le C\sqrt d,\qquad
 v_n=\frac{\operatorname{Var}(D)}{d},\qquad
 c+1-v_n\ge\kappa,
\end{equation}
\begin{equation}\label{eq:range-p}
 d\longrightarrow\infty,\qquad d=o(n^{1/7}).
\end{equation}
Set
\begin{equation}\label{eq:bk}
 b=1-\pi\frac{\mathbb E[D(D-1)]}{d},\qquad
 K=(1-2\pi)\frac{\mathbb E[D(D-1)(D-2)]}{d},\qquad
 A=\frac{b^2}{2K}.
\end{equation}
For all sufficiently large $n$, the degree sequence is graphical,
$0<\pi<1$, $b>0$ and $K>0$. On the uniform simple graph with these
degrees, under the threshold and seed conventions above, there exist
finite constants $B_-,B_+$ depending only on $c,C,\kappa,\gamma$ such
that, for every sequence of prescribed deterministic seed sizes,
\begin{align}
 a\le(1-\gamma)nA
 &\ \Longrightarrow\quad
 \Pr\bigl(|A_\infty|\le B_-n/d^3\bigr)\longrightarrow1,
 \label{eq:sub}\\
 a\ge(1+\gamma)nA
 &\ \Longrightarrow\quad
 \Pr\bigl(n-|A_\infty|\le B_+\mathcal R_{\rm P}(n,d,C)\bigr)
       \longrightarrow1.\label{eq:super}
\end{align}
Here and below an inequality on $a_n$ is required eventually.
No convergence of $v_n$ is assumed.
\end{theorem}

\begin{theorem}[E: independent edges]\label{thm:er}
Fix $c>0$ and $\gamma\in(0,1)$, and let
\begin{equation}\label{eq:range-e}
 d\longrightarrow\infty,\qquad \frac{d^5}{n}\longrightarrow0.
\end{equation}
Take $G\sim G(n,d/n)$ with the threshold and seed conventions above,
and put $A_{\rm E}=c^2/(2d^4)$ and $a_{\rm E}=nA_{\rm E}$.
There exist finite constants $B_-^{\rm E},B_+^{\rm E}$ depending only
on $c,\gamma$ such that, for every sequence of prescribed deterministic
seed sizes,
\begin{align}
 a\le(1-\gamma)a_{\rm E}
 &\ \Longrightarrow\quad
 \Pr\bigl(|A_\infty|\le B_-^{\rm E}n/d^3\bigr)\longrightarrow1,
 \label{eq:sub-e}\\
 a\ge(1+\gamma)a_{\rm E}
 &\ \Longrightarrow\quad
 \Pr\bigl(n-|A_\infty|\le B_+^{\rm E}\mathcal R_{\rm E}(n,d)\bigr)
       \longrightarrow1.\label{eq:super-e}
\end{align}
This theorem treats the independent-edge model separately, without
the prescribed-degree assumptions of Theorem~\ref{thm:v}.
Here $d=np$ parametrizes the edge probability, and the exact expected
degree is $(n-1)d/n$.
\end{theorem}

Both residual bounds are $o(n)$ in their respective ranges, so the
supercritical conclusions imply $|A_\infty|/n\to1$ in probability.
These general conclusions concern almost complete activation; the regular
case is strengthened to full activation in Corollary~\ref{cor:regular}.
No shrinking critical window is asserted.
The seed-size quantifiers concern prescribed sizes and independent
uniform seeds, not seeds selected after observing graph or thresholds.

\paragraph{Constants.}
The model parameters $c,C,\kappa$ and relative margin $\gamma$ are
fixed whenever applicable. Constants in the prescribed-degree probability
bounds may depend only on $c,C,\kappa,\gamma$; those in the independent-edge
bounds may depend only on $c,\gamma$.
Constants in the logarithmic corollary may additionally depend on its
fixed $K_0$. Unspecified constants may vary between estimates, while the
model parameters remain fixed. The index from which an asymptotic
estimate holds may depend on the chosen sequence and its convergence
rates; it need not be uniform over those sequences.
In the abstract deterministic barrier lemma, constants may also depend
on its uniform comparison and derivative bounds and its fixed local
interval. None of these statements permits $\gamma\to0$ or
$\kappa\to0$.

\begin{lemma}[Empirical moments and transfer at fixed margins]\label{lem:actual-v}
Under \eqref{eq:degrees}, with fixed $c,C,\kappa$ and $d\to\infty$,
let $q_n=c+1-v_n$ and $A_{0,n}=q_n^2/(2d^4)$.
For $b,K,A$ in \eqref{eq:bk}, uniformly over these empirical laws,
\begin{equation}\label{eq:scale}
 b=\frac{q_n}{d}+\frac{c(v_n-1)}{d^2},\qquad
 K=d^2\bigl(1+O(d^{-1})\bigr),\qquad
 A=\frac{q_n^2}{2d^4}\bigl(1+O(d^{-1})\bigr).
\end{equation}
Thus \eqref{eq:sub} and \eqref{eq:super} remain valid with $A_{0,n}$
in place of $A$, for each fixed relative margin (with possibly changed
constants). If in addition $v_n\to v$, then $v<c+1$ and
\begin{equation}\label{eq:scale-limit}
 nA\sim\frac{n(c+1-v)^2}{2d^4},
\end{equation}
and that limiting expression may likewise replace $nA$ in the
fixed-margin conclusions. The deterministic estimates in
\eqref{eq:scale} do not require a relation between $n$ and $d$;
the probability conclusions retain \eqref{eq:range-p}.
\end{lemma}

\begin{corollary}[Regular full activation and separation]\label{cor:regular}
Fix $c>0$. Let $d=d_n$ be integer, $nd$ even, and assume
$d\to\infty$ and $d=o(n^{1/7})$. Compare the uniform simple
$d$-regular graph and $G(n,d/n)$ under the same threshold and seed
conventions. Their leading seed scales are
\[
 t_{\rm R}=\frac{(c+1)^2n}{2d^4},\qquad
 t_{\rm E}=\frac{c^2n}{2d^4},\qquad
 \frac{t_{\rm R}}{t_{\rm E}}=\frac{(c+1)^2}{c^2}.
\]
Precisely, for every fixed $\gamma\in(0,1)$, seeds below
$(1-\gamma)t_{\rm R}$ give $O(n/d^3)$ final active vertices in the
regular model, whereas $a\ge(1+\gamma)t_{\rm R}$ gives
$A_\infty^{\rm R}=[n]$, with high probability.
In the E model, seeds below $(1-\gamma)t_{\rm E}$ give
$O(n/d^3)$ active vertices, whereas seeds above $(1+\gamma)t_{\rm E}$
leave $O(\mathcal R_{\rm E}(n,d))$ inactive vertices, with high probability.
For every fixed $c^2/2<\lambda<(c+1)^2/2$ and deterministic integer
$a\sim\lambda n/d^4$,
\[
 |A_\infty^{\rm R}|=O(n/d^3),\qquad
 n-|A_\infty^{\rm E}|=O(\mathcal R_{\rm E}(n,d))=o(n)
 \quad\hbox{with high probability}.
\]
If also $d\ge K_0\log n$ eventually for fixed $K_0>2/3$, the E
residual in this separation is $O(n^{2/3})$ with high probability.
The comparison uses the same asymptotic mean degree; the exact means
are $d$ and $d-d/n$, respectively.
\end{corollary}

The proofs of Lemma~\ref{lem:actual-v} and Corollary~\ref{cor:regular}
are given in Section~\ref{sec:consequences}. That section also records
sufficient logarithmic conditions for an $O(n^{2/3})$ inactive
remainder in the general prescribed-degree and independent-edge models,
and an example with two degree values.

\subsection{Relation to earlier work}
Section VI, equation (24) of the arXiv:1003.5583v2 (2010)
version of \citet{baxter} gives the distributional
threshold equation used here. Set its damage parameter to one and its
threshold probabilities at zero, one, and two to
$\alpha,(1-\alpha)\pi,(1-\alpha)(1-\pi)$.
The coefficients in \eqref{eq:bk} are Taylor coefficients of this
established equation. We use the equation and its formal expansion as
existing tools; our contribution is the quantitative terminal control
specified above.

For independent edges, \citet{janson2012} use independent Bernoulli
marks, and \citet{ballen} allow nonuniform thresholds. The latter's
exact binomial response gives our mixed drift. Their Theorem~1 assumes
a fixed positive gap $\beta$ through $d\zeta_1\le1-\beta$,
writing $d$ for their expected-degree parameter $\eta\phi$.
In our regime $d\pi=1-c/d$, this gap vanishes, so that theorem
does not directly apply. For a fixed relative margin $\gamma>0$,
formal substitution of $r_m=2$, $\beta=c/d$ and $t_*\sim cn/d^3$
into its displayed error bound gives $O_{c,\gamma}(d^9/n)$;
this substitution alone does not establish a uniform theorem for a
vanishing gap. Here we apply the distribution-free forward/reverse
empirical-process proof in Lemmas~7.2--7.3 of arXiv:1012.3535v2
to the mixed activation-time CDF, localize the early horizon and use
dyadic error budgets, giving passage failure $O(d^5/n)$.
This is a specialization of established arguments, rather than a
new exploration or maximal inequality.

For prescribed degrees, \citet{amini2010} requires a finite limiting mean
and a bounded empirical second moment. \citet{lelarge} also treats
independent random thresholds on prescribed-degree graphs, including
uniform simple graphs. Condition~1 and Theorem~10 in the
arXiv:1012.2062v4 (2011) version assume a fixed limiting degree law
with finite mean and $\sum_i d_i^2=O(n)$; these assumptions exclude
our growing-degree sequences. The $o_p(n)$ exploration approximation
in its proof does not by itself resolve our bottleneck of order
$n/d^3$. Condition 1.1 in the
arXiv:2103.00357v2 (2022) version of \citet{abc} requires a limiting
degree/threshold law and uniform exponential
moments. These hypotheses fail in \eqref{eq:degrees}. We use the
exponential half-edge exploration described by these authors. Our
quantitative control includes local loss variance, stopped short-time
predictable quadratic variation, the fixed-size seed mean and time
interpolation. \citet{garetto} already discuss degree
heterogeneity and compare regular and two-level examples; their
condition excluding one-edge activation and their homogeneous
prescribed-degree calculation do not directly give our mixture.

\citet{janson2009} includes fixed-size initial sets and a critical-window
result for fixed regular degree and common threshold. Prescribed seed
size is therefore already covered in the literature. A different
comparison concerns finite-time approximation:
\citet{rossi} develop quantitative finite-time
coupling and concentration estimates for threshold dynamics
(Sections~4.2--4.3.1 of the 2017 accepted manuscript).
We prove below the excess-degree version needed here. Under our
prescribed-degree assumptions and fixed supercritical seed margin,
this adaptation yields an $O(n^{2/3})$ inactive remainder when
$d\log d=o(\log n)$;
see Appendix~\ref{app:rossi-overlap}. This uses the excess-degree
offspring law, joint concentration of pairing and labels,
and conditioning on simplicity. In the comparison established here,
the subcritical terminal upper bound still requires control of
activation after the certified finite horizon. The adaptation neither
covers our full degree range nor proves the independent-edge counterpart.

For fixed degree $d$ and a common threshold $2\le k\le d-1$,
\citet[Theorems~1--3]{balogh} establish threshold and window estimates,
including deterministic initial sets. Their equation~(3) and
equations~(17)--(20) already express the threshold through the
infimum of a deterministic response ratio. At the level of the
deterministic response, replacing their binomial tail by the mixture
of the tails for thresholds one and two gives our regular response;
this substitution does not itself extend their probability theorems.
Their passage from configurations to simple graphs uses a simplicity
probability bounded away from zero for fixed $d$; it does not directly
supply the uniform conditioning argument needed here.

\paragraph{Organization.}
Section~\ref{sec:consequences} proves the consequences stated above,
gives further examples, and establishes the common deterministic barrier.
Sections~\ref{sec:v} and \ref{sec:er} connect it to the respective
explorations and concentration bounds; the prescribed-degree proof
then conditions on simplicity. Section~\ref{sec:sandwich} explains
the precision obstruction to static inclusion transfer.
Section~\ref{sec:finite-degree} derives finite-degree and higher-moment
corrections, and Section~\ref{sec:scope} records the remaining limits.
The appendices supply the local concentration estimates and the
slow-degree comparison with finite-time approximation.
Appendix~\ref{app:bounded-tail} adds the bounded-tail fixed-margin
extension, with its own threshold law and scope.

\section{Consequences and deterministic estimates}\label{sec:consequences}
\subsection{Empirical moments and consequences}
\begin{proof}[Proof of Lemma~\ref{lem:actual-v}]
Write $\mu_3=\mathbb E[(D-d)^3]$. The deviation bound gives
$0\le v_n\le C^2$ and
$|\mu_3|\le C\sqrt d\,\mathbb E[(D-d)^2]=O(d^{3/2})$.
The exact factorial-moment identities give the displayed formula for
$b$ and
\[
 K=(1-2\pi)\left[d^2+(3v_n-3)d+2-3v_n+\frac{\mu_3}{d}\right].
\]
Consequently $K/d^2=1+O(d^{-1})$ and, since $q_n\ge\kappa$,
$b=(q_n/d)(1+O(d^{-1}))$. This proves \eqref{eq:scale}.
For any positive densities $T_n,A_n$ with $T_n/A_n\to1$ and fixed
$\gamma\in(0,1)$, eventually
\[
 (1-\gamma)T_n\le(1-\gamma/2)A_n,\qquad
 (1+\gamma)T_n\ge(1+\gamma/2)A_n.
\]
Apply the appropriate theorem with margin $\gamma/2$.
If $v_n\to v$, the fixed gap gives $c+1-v\ge\kappa$, and the
same argument proves the limiting specialization.
\end{proof}

\paragraph{Integer seed sizes.}\label{par:seed-transfer}
In both \eqref{eq:range-p} and \eqref{eq:range-e}, $n/d^4\to\infty$.
Thus rounding a size of order $n/d^4$ changes it by a relative $o(1)$.
More precisely, for any theorem centre $A_\star\asymp d^{-4}$,
$\lfloor(1-\gamma)nA_\star\rfloor$ satisfies its subcritical
inequality and $\lceil(1+\gamma)nA_\star\rceil$ satisfies its
supercritical inequality. The reference size
$\lceil(1+\gamma/2)nA_\star\rceil$ is no larger than every permitted
supercritical size for all sufficiently large $n$.
Uniform seeds of all prescribed sizes can be coupled as initial
segments of one uniform permutation independent of graph and thresholds.
Monotonicity of bootstrap closure then transfers a bound from a reference
size to each permitted smaller or larger prescribed size.

\begin{corollary}[A sufficient logarithmic condition]\label{cor:logarithmic}
In the supercritical conclusion of Theorem~\ref{thm:v}, suppose also
$d\ge K_0\log n$ eventually for some fixed $K_0>1/(3\log2)$.
Then $n-|A_\infty|=O(n^{2/3})$ with high probability.
In the supercritical conclusion of Theorem~\ref{thm:er}, the same
conclusion holds if $d\ge K_0\log n$ eventually for some fixed
$K_0>2/3$.
These are sufficient conditions, not necessary conditions or optimal
constants; the boundary equalities are not asserted.
\end{corollary}
\begin{proof}
For P the function
$f(x)=x\log2-C\sqrt x\log2-\log x$ is increasing for all sufficiently
large $x$. Hence
\[
 f(d)-\tfrac13\log n
 \ge (K_0\log2-\tfrac13)\log n
       -O(\sqrt{\log n})-O(\log\log n)\longrightarrow\infty.
\]
It follows that $d2^{-d+C\sqrt d}=o(n^{-1/3})$.
For E, $g(x)=x/2-\log(1+x)$ is eventually increasing and $d/n=o(1)$.
Therefore
\[
 g(d)-\tfrac13\log n-d/n
 \ge (K_0/2-\tfrac13)\log n-O(\log\log n)-o(1)
       \longrightarrow\infty,
\]
so $(1+d)e^{-d/2+d/n}=o(n^{-1/3})$.
Insert these estimates into \eqref{eq:residuals} and the corresponding
supercritical theorem.
\end{proof}

\begin{lemma}[Completion from half activation]\label{lem:regular-completion}
Let $G$ be a simple $d$-regular graph on $n\ge2$ vertices, with
adjacency matrix $M$ and eigenvalues
$\lambda_1(M)\ge\cdots\ge\lambda_n(M)$. Put
$\lambda_* = \max\{\lambda_2(M),-\lambda_n(M)\}$.
If $d-\lambda_*>2$, then any bootstrap process with thresholds in
$\{1,2\}$ that reaches at least $n/2$ active vertices eventually
activates all vertices.
\end{lemma}
\begin{proof}
Let $U$ be the final inactive set and $m=|U|\le n/2$.
Stability gives $e(U,U^c)\le m$, where $e(U,U^c)$ counts edges
across the cut once. For $x=\mathbf1_U-(m/n)\mathbf1$, we have
$x\perp\mathbf1$ and $\|x\|^2=m(1-m/n)$, so the Rayleigh bound gives
\[
 e(U,U^c)=x^{\mathsf T}(dI-M)x
 \ge(d-\lambda_*)m(1-m/n).
\]
If $m>0$, the right-hand side is greater than $m$, a contradiction.
\end{proof}

\begin{proof}[Proof of Corollary~\ref{cor:regular}]
For the regular model $C=v_n=0$ and
\[
 b_{\rm R}=\frac{c+1}{d}-\frac c{d^2},\qquad
 K_{\rm R}=(1-2\pi)(d-1)(d-2),\qquad nA_{\rm R}\sim t_{\rm R}.
\]
Theorem~\ref{thm:v} and Lemma~\ref{lem:actual-v}, with a reduced fixed
margin, give the regular subcritical statement and imply
$\Pr(|A_\infty^{\rm R}|\ge n/2)\to1$ for
$a\ge(1+\gamma)t_{\rm R}$.
In this uniform simple model, \citet[Theorem~1.1]{cook2018} give, for $1\le d\le n^{2/3}$,
$\Pr(\lambda_*\le C_{\rm sp}\sqrt d)\ge1-n^{-1}$ with a fixed
$C_{\rm sp}$ for all sufficiently large $n$.
Their degree range includes the present one, and $d\to\infty$ implies
$d-\lambda_*>2$ on this spectral event eventually.
Intersecting the two events and applying Lemma~\ref{lem:regular-completion}
proves full activation; independence of the events is not needed.
The common range implies \eqref{eq:range-e}, since
\[
 \frac{d^5}{n}
  =\left(\frac d{n^{1/7}}\right)^5n^{-2/7}
       \longrightarrow0.
\]
Thus Theorem~\ref{thm:er} applies as well. Put
$L=c^2/2$, $U=(c+1)^2/2$ and
\[
 \gamma_* =\frac14\min\{1,\lambda/L-1,1-\lambda/U\}>0.
\]
The ratios $a/t_{\rm E}\to\lambda/L$ and
$a/(nA_{\rm R})\to\lambda/U$ place the seeds on opposite sides
with margin $\gamma_*$. The theorem bounds and a union bound give
the joint conclusion. Independence between the two graphs is not
needed: one may use any coupling of their laws, with common threshold
labels and a uniform seed permutation independent of the graph pair.
Rounding $\lambda n/d^4$ is harmless by $n/d^4\to\infty$.
The last assertion follows from Corollary~\ref{cor:logarithmic}.
\end{proof}

\begin{corollary}[Two degree values]\label{cor:two}
Fix $c>0$, $\kappa>0$ and $\gamma\in(0,1)$. Let $n$ be even,
$s\ge0$, and let $d-s,d+s$ be nonnegative integers with $nd$ even.
Assign degree $d-s$ to half the vertices and $d+s$ to the other half.
Assume $d\to\infty$, $d=o(n^{1/7})$, and
\[
 v_n=\frac{s^2}{d},\qquad q_n=c+1-v_n\ge\kappa.
\]
Set $t_{2,n}=nq_n^2/(2d^4)$ and $C_2=\sqrt{c+1}$.
The sequence is graphical eventually. For the uniform simple graph
with these degrees and the stated independent threshold and seed model,
\[
 \begin{split}
 a\le(1-\gamma)t_{2,n}&\ \Longrightarrow\quad
       |A_\infty|=O(n/d^3)\quad\hbox{with high probability},\\
 a\ge(1+\gamma)t_{2,n}&\ \Longrightarrow\quad
       n-|A_\infty|=O(\mathcal R_{\rm P}(n,d,C_2))
       \quad\hbox{with high probability}.
 \end{split}
\]
If $d\ge K_0\log n$ eventually for fixed $K_0>1/(3\log2)$,
the supercritical residual is $O(n^{2/3})$ with high probability.
If additionally $v_n\to v<c+1$, then $t_{2,n}$ may be replaced in
these fixed-margin statements by $n(c+1-v)^2/(2d^4)$.
\end{corollary}
\begin{proof}
The empirical mean is exactly $d$, its variance is $s^2$, and its
third central moment is zero. The gap implies
$s\le\sqrt{c+1}\sqrt d$, so the P deviation condition holds with
fixed $C_2$. In particular,
\[
 b=\frac{q_n}{d}+\frac{c(v_n-1)}{d^2},\qquad
 K=(1-2\pi)[d^2+(3v_n-3)d+2-3v_n].
\]
Apply Theorem~\ref{thm:v}, Lemma~\ref{lem:actual-v}, and
Corollary~\ref{cor:logarithmic}.
\end{proof}

The conditions concern the two degrees themselves: $d,s$ may both
be half-integers, and even $n$ alone does not ensure even $nd$.
For a concrete limiting-variance example, fix $0\le v<c+1$, take
integer $d\to\infty$ and even $n$ in the stated range, and set
$s=\lfloor\sqrt{vd}\rfloor$. Then the degrees are nonnegative
integers eventually, $nd$ is even, and
\[
 0\le v-v_n\le 2\sqrt{v/d},\qquad v_n\longrightarrow v.
\]
A fixed positive gap follows, and the limiting specialization applies.
For finite-degree corrections the actual $v_n=s^2/d$ must still be
retained; this example supplies no $o(d^{-1})$ rate for $v_n-v$.
The case $s=0$ is regular, so Corollary~\ref{cor:regular} strengthens
its supercritical conclusion to full activation. The agreement of a leading coefficient
at the formal variance value $v=1$ does not make the independent-edge
model a special case of the prescribed maximum-deviation hypothesis.

\subsection{A common deterministic barrier}\label{sec:common}
The two models share the same local mechanism: an unseeded response
$F$ with $F'(0)=1-b$ and $F''(0)=K$ produces a barrier of order
$b^2/(2K)$.
The next lemma controls this barrier and the drift beyond it; the
probabilistic arguments will then compare the exploration queues with
that drift.
\begin{lemma}[A small-deficit response barrier]\label{lem:barrier}
Let $d\to\infty$ and let $F:[0,1/2]\to[0,1)$ be nondecreasing,
with $F(0)=0$. Suppose that $F$ is $C^3$ on $[0,\varepsilon/d]$
for a fixed $\varepsilon>0$, with one-sided endpoint derivatives,
and that
\[
 F'(0)=1-b,\quad F''(0)=K,\quad
 b_*/d\le b\le b^*/d,\quad K_*d^2\le K\le K^*d^2,
\]
\[
 F''\ge K/2,\quad |F'''|\le C_3d^3
       \quad(0\le z\le\varepsilon/d).
\]
Here $b_*,b^*,K_*,K^*,C_3$ are fixed positive constants.
Assume $\sup_{\xi\le z\le1/2}(1-F(z))=o(1)$ for each fixed
$\xi\in(0,1/2)$. Put $A=b^2/(2K)$, $z_0=b/K$ and
$h_\alpha(z)=\alpha+(1-\alpha)F(z)-z$.
Choose any fixed $L>8b^*/K_*$. For every fixed $\eta\in(0,1)$,
if $\alpha_\pm=(1\pm\eta)A+o(A)$, then, eventually, for a fixed
$k_\eta>0$,
\begin{equation}\label{eq:common-margins}
 h_{\alpha_+}(z)\ge k_\eta
 \begin{cases}
 d^{-4},&0\le z\le L/d^3,\\
 d^2z^2,&L/d^3\le z\le\varepsilon/d,\\
 1,&\varepsilon/d\le z\le1/2,
 \end{cases}
 \qquad h_{\alpha_-}(z_0)\le-\eta A/2.
\end{equation}
The constants depend only on the displayed fixed bounds, $L$ and
$\eta$. Moreover $R=(z-F)/(1-F)$ has a unique global maximum
on $[0,1/2]$, attained in $(0,8b/K)$, of value $A+o(A)$.
\end{lemma}
\begin{proof}
Uniformly on $[0,L/d^3]$, for $\alpha=O(A)$,
\begin{equation}\label{eq:common-square}
 h_\alpha(z)=\alpha-A+\frac K2(z-z_0)^2
                +O(d^3z^3+\alpha F(z)).
\end{equation}
Indeed Taylor expansion also gives $F(z)=O(d^{-3})$ there, so
the two remainder terms are $O(d^{-6})$ and $O(d^{-7})$,
both $o(A)$. This proves the early and negative margins.
For $z\ge L/d^3\ge8b/K$ and $z\le\varepsilon/d$,
\[
 F(z)-z\ge-bz+Kz^2/4\ge Kz^2/8.
\]
Thus $F(\varepsilon/d)\ge q_0$ eventually for some fixed
$0<q_0\le1/4$. On $[\varepsilon/d,q_0/2]$ monotonicity gives
$F-z\ge q_0/2$; on $[q_0/2,1/2]$ the tail hypothesis gives
$F-z\ge1/4$ eventually. Since
$h_\alpha=F-z+\alpha(1-F)\ge F-z$ for $\alpha\ge0$,
this proves the remaining margins.

For $B=1-F$, on $[0,8b/K]$,
\[
 R'=g/B^2,\quad g=B-(1-z)F',\quad g'=-(1-z)F''.
\]
Here $g$ is strictly decreasing, $g(0)=b$ and
$g(8b/K)=-7b+O(d^{-3})<0$. There is a unique early maximum.
The preceding middle and late argument, starting at $8b/K$,
gives $F-z>0$ everywhere outside that interval, so $R<0$ there.
Inside it $R=bz-Kz^2/2+o(A)$ uniformly; its maximum is $A+o(A)>0$.
All perturbations $o(A)$ are absorbed by the fixed margin.
\end{proof}

\paragraph{Verification for the two responses.}
For the prescribed empirical law define, with $y=1-z$,
\[
 B_k(z)=y^{k-1}+(1-\pi)(k-1)zy^{k-2},\qquad
 F_{\rm P}(z)=1-\frac1d\mathbb E[D B_D(z)].
\]
For independent edges define
\[
 F_{\rm E}(z)=1-e^{-dz}[1+(1-\pi)dz].
\]
The P verification uses only \eqref{eq:degrees}, fixed $c,C,\kappa$
and $d\to\infty$. Lemma~\ref{lem:actual-v} gives its $b,K,A$.
Eventually $\delta=\min_i d_i\ge7$, $\Delta=\max_i d_i\asymp d$.
For $F_k=1-B_k$ and $1\le j\le5$,
\[
 F_k^{(j)}(z)=(-1)^{j+1}(k-1)_j y^{k-2-j}
 \{[1-j(1-\pi)]y+(1-\pi)(k-1-j)z\}.
\]
In particular $F_k'\ge0$ and
\[
 F_k''=(k-1)(k-2)y^{k-4}
 \{(1-2\pi)y-(1-\pi)(k-3)z\}.
\]
For E, for $1\le j\le5$,
\[
 F_{\rm E}^{(j)}(z)=(-1)^{j+1}d^je^{-dz}
       [1-j(1-\pi)+(1-\pi)dz].
\]
It has $b=c/d$, $K=(1-2\pi)d^2$ and
$A^{\rm quad}=b^2/(2K)=A_{\rm E}(1+O(d^{-1}))$.
Both responses are nondecreasing and strictly below one on $[0,1/2]$.
For a sufficiently small fixed $\varepsilon>0$ the formulas imply
$F''\ge K/2$ and $|F^{(j)}|=O(d^j)$, $2\le j\le5$,
on $[0,\varepsilon/d]$. For P, $k/d=1+O(d^{-1/2})$ uniformly,
so the ratio of each second derivative to its value at zero is
at least $1/2$ for small $\varepsilon$. For E its ratio is
$e^{-dz}[1-(1-\pi)dz/(1-2\pi)]$, with the same property.
For each fixed $\xi>0$ the tail bounds are
\[
 1-F_{\rm P}(z)\le(1+\Delta)(1-\xi)^{\delta-2},\qquad
 1-F_{\rm E}(z)\le(1+d)e^{-d\xi}
       \quad(\xi\le z\le1/2).
\]
Thus Lemma~\ref{lem:barrier} applies to both responses without a
relation between $n$ and $d$. The same formulas give $F'=O(1)$
on each fixed early interval $[0,L/d^3]$.

\section{Proof for prescribed degrees}\label{sec:v}
Write $\delta=\min_i d_i$, $\Delta=\max_i d_i$, and take $n$ large
enough that $\delta\ge4$ and $0<\pi<1$. All constants below depend
only on the fixed parameters of Theorem~\ref{thm:v}.
Put $M=\sum_i d_i=nd$. We first work with a uniform pairing of
the $M$ labelled half-edges, counting edges with multiplicity. We compare the active queue with its
deterministic drift, use concentration to locate stopping or passage,
and then bound the final number of active or inactive vertices.
Finally, we condition on simplicity and check that the failure
probabilities remain $o(1)$.

\paragraph{Response and deterministic margins.}
Use $F=F_{\rm P}$ from Section~\ref{sec:common}, $B=1-F$,
$M=nd$ and $h_\alpha=\alpha+(1-\alpha)F-z$.
Lemma~\ref{lem:actual-v} gives $b\asymp d^{-1}$, $K\asymp d^2$
and $A\asymp d^{-4}$; the present degree range gives $nA\to\infty$.
Take
\[
 a_- =\lfloor(1-\gamma)nA\rfloor,\qquad
 a_+ =\lceil(1+\gamma/2)nA\rceil,\qquad \alpha_\pm=a_\pm/n.
\]
Their rounding errors are $o(A)$. Choose $L$ as in
Lemma~\ref{lem:barrier} and apply it
with margins $\gamma$ and $\gamma/2$ respectively. Put
$\eta=\varepsilon$, $u=-\log(1-z)$ and
$D_\alpha(u)=M(1-z)h_\alpha(z)$. Since $1/2\le1-z\le1$,
for a sufficiently small fixed $k>0$,
\begin{equation}\label{lup:margins}
 D_{\alpha_+}(u)\ge
 \begin{cases}
 k n/d^3,&0\le z\le L/d^3,\\
 k nd^3z^2,&L/d^3\le z\le\eta/d,\\
 k nd,&\eta/d\le z\le1/2,
 \end{cases}
 \qquad D_{\alpha_-}(u_*)\le-kn/d^3,
 \quad u_*=-\log(1-b/K).
\end{equation}

\paragraph{Exploration and stopping convention.}
Give every original half-edge an independent rate-one exponential clock,
independently of the seeds and thresholds. Reserve one active end,
without inspecting its future clock, and pair it to the first ringing
unreserved end. Record a mark at the partner vertex, activate that
vertex if its threshold is reached, and reserve the next active end.
Use only revealed information to choose a reserved end. Stop at time
$\tau$ when no active end remains after a pairing. Complete all remaining
pairs uniformly. By exponential memorylessness each exposed partner
is uniform among the other remaining ends. Induction on the number
of remaining ends therefore gives a uniform full pairing conditional
on every seed and threshold assignment, so this pairing is independent
of these assignments.

Every activation is justified by exposed edges from active vertices.
At stopping all ends of active vertices have been explored; the remaining
pairs join inactive vertices, each of which has received fewer than its
threshold many marks from active vertices. The final set is therefore
closed under the activation rule, and is exactly the bootstrap closure.
In particular completion adds no activations.

Let $X_i(u)$ count original clocks at vertex $i$ expiring by $u$,
including latent clocks of removed ends, and set
\[
 H(u)=\sum_{i\notin A_0}(d_i-X_i(u))\mathbf1_{\{X_i(u)<R_i\}}.
\]
Up to and including $\tau$, this is the actual number of inactive ends:
before activation no end at that vertex was manually removed, and after
activation its original-clock count stays at least its threshold.
The formula remains a latent process after $\tau$, without implying
any further real activation. Its unconditional mean is exactly
\begin{equation}\label{lup:mean}
 \bar H(u)=(1-\alpha)M e^{-u}B(1-e^{-u}),\qquad
 Me^{-2u}-\bar H(u)=D_\alpha(u).
\end{equation}
Let $W$ count unreserved ends, initially $W(0)=M-1$.
Each pairing decreases $W$ by two, with a fictitious final reservation
if no active end is available. Thus the bookkeeping value of $W$ at
$\tau$ is one less than the actual remaining count, and
\begin{equation}\label{lup:stop}
 Q(u):=W(u)-H(u)\ge0\quad(u<\tau),\qquad Q(\tau)=-1.
\end{equation}
All queue comparisons below include this terminal convention.

\paragraph{Local passage estimate.}
Put $U=\log2$. The following estimate compares the queue with the
three drift scales in \eqref{lup:margins}.
\begin{lemma}[Local concentration of the active queue]\label{lem:lup-local}
Under the hypotheses of Theorem~\ref{thm:v}, consider the configuration
model exploration just defined with either reference size $a_-$ or $a_+$.
Let $k,L,\eta$ be as in \eqref{lup:margins}. There are constants
$C_2,k_2>0$, depending only on $c,C,\kappa,\gamma$, such that, except
on an event of probability at most
\begin{equation}\label{lup:localfail}
 C_2(1+d^2)(1+\log d)\exp(-k_2 n/d^5),
\end{equation}
the following bounds hold simultaneously for $0\le u\le U\wedge\tau$,
where $z=1-e^{-u}$:
\[
 |Q(u)-D_\alpha(u)|<\frac{k}{2}
 \begin{cases}
 n/d^3,&0\le z\le L/d^3,\\
 nd^3z^2,&L/d^3<z\le\eta/d,\\
 nd,&\eta/d<z\le1/2.
 \end{cases}
\]
For size $a_-$, on the same event, if $u_*\le\tau$ then
$|Q(u_*)-D_{\alpha_-}(u_*)|<kn/(2d^3)$.
\end{lemma}
The proof in Appendix~\ref{lup:appendix} uses separate vertex, seed and stopped-queue estimates at the three drift scales.

For $a_-$, if $\tau\ge u_*$ the queue
would be less than $-1$ at $u_*$, contradicting \eqref{lup:stop}.
Thus $\tau<u_*$. For $a_+$, stopping at any $\tau\le U$ would give
$Q(\tau)>0$, again a contradiction. Thus $\tau>U$.

\paragraph{Counts at the two outcomes.}
In the subcritical case every eventually active nonseed vertex satisfies
$X_i(u_*)\ge R_i$. The indicators of these events over all vertices are
independent under the unconditioned clock and threshold law, and their
sum has mean at most
\[
 \mu_*\le\sum_i\left[\pi d_i z_*+\binom{d_i}{2}z_*^2\right]
 \le C_3 n/d^3,\qquad z_*=b/K.
\]
Bernstein's inequality for Bernoulli sums, with deviation $n/d^3$,
bounds the probability that this sum exceeds $(C_3+1)n/d^3$ by
$\exp(-k_3 n/d^3)$. Adding $a_-=O(n/d^4)$ proves the required count
on the intersection with the stopping event. Independence of these
two events is not required.

For the supercritical case put
$N_U=\sum_i\mathbf1_{\{X_i(U)<R_i\}}$.
These indicators are independent under the unconditioned clock and
threshold law, with
\[
 \mu_U=\mathbb EN_U
 =\sum_i2^{-d_i}[1+(1-\pi)d_i]\le n(1+\Delta)2^{-\delta}.
\]
Hoeffding gives
\begin{equation}\label{lup:latefail}
 \Pr(N_U>\mu_U+n^{2/3})\le e^{-2n^{1/3}}.
\end{equation}
Take $k_4=2$ in \eqref{lup:totalfail}.
On $\{\tau>U\}$ every ultimately inactive vertex has
$X_i(U)<R_i$, since no end of an inactive vertex was manually removed.
Intersecting passage and this unconditional count event therefore gives
\[
 n-|A_\infty|\le N_U
 \le\mu_U+n^{2/3}
 \le B_+n[d2^{-d+C\sqrt d}+n^{-1/3}].
\]

\paragraph{Simplicity, graphicality and uniformity.}
For $M_r=\sum_i(d_i)_r$ and $\nu=M_2/M=d-1+v_n$, Lemma 5.1 of
\citet{mckay} gives, in the present regime,
\[
 \log s=-\nu/2-\nu^2/4+O(d^3/n),\qquad
 s=\Pr\{\text{the pairing is simple}\}\ge e^{-C_0d^2}.
\]
Indeed its condition $\Delta=o(M^{1/3})$ follows from
$\Delta^3/M=O(d^2/n)=o(1)$. Its additional terms are
$O(\Delta^4/M)$ and its remainder is $O(\Delta^3/M)$, by
$M_r\le M\Delta^{r-1}$; this yields the displayed specialization.
The positive probability also proves eventual graphicality.
Every simple graph has exactly $\prod_i d_i!$ pairings, so the
conditional graph is uniform simple. The full pairing is independent
of thresholds and seeds, whose joint law is therefore preserved.
The clocks need not stay independent after conditioning.

Divide the sum of all applicable bad-event probabilities by $s$.
For either reference outcome a common upper envelope is
\begin{equation}\label{lup:totalfail}
 e^{C_0d^2}\left[
 C_2(1+d^2)(1+\log d)e^{-k_2n/d^5}
 +e^{-k_3n/d^3}+4e^{-k_4n^{1/3}}\right]=o(1).
\end{equation}
This uses only
\[
 \frac{n/d^5}{d^2}=\frac n{d^7}\to\infty,\qquad
 \frac{n/d^5}{\log n}\ge\frac{n^{2/7}}{\log n}\to\infty,
 \qquad \frac{d^2}{n^{1/3}}\to0,
\]
where the middle inequality holds eventually. In particular it does
not require $d^2\gg\log n$.
Every permitted subcritical size is at most $a_-$, and every permitted
supercritical size is at least $a_+$ eventually, since $nA\to\infty$.
The independent permutation coupling in the paragraph on integer seed sizes in Section~\ref{par:seed-transfer}
transfers the reference conclusions by monotonicity. The empty and
full seed cases are immediate. This completes the proof.

\section{Proof for independent edges}\label{sec:er}
We specialize the independent-mark construction and the distribution-free empirical-process argument of \cite{janson2012} to the threshold mixture.

Put $p=d/n$ and $N=n-a$. All assertions below concern sufficiently
large $n$, so $0<p<1/2$ and $0<\pi<1$.

Condition on a particular seed set of size $a$. At each real step $i$,
choose one active, unused vertex $u_i$ by a deterministic rule based
only on the revealed history and thresholds. Reveal its edges to all
other unused vertices. These edges are fresh: neither endpoint was
previously used. Conditional on the history and all thresholds, they
are independent Bernoulli$(p)$ variables. For each nonseed vertex $v$,
set $I_v(i)$ equal to this edge indicator when $v$ is a target of the
step; in every other cell (including $v=u_i$) insert a fresh independent
Bernoulli$(p)$ variable. Thus each new column has product Bernoulli$(p)$
law independent of all preceding columns and thresholds. After real
exploration stops, fill the remaining columns with fresh variables.
Unexposed graph edges are completed independently. Deferred decisions
therefore give exactly the required graph law, independent of the
thresholds, and an iid mark array independent of the thresholds.
The artificial cells are auxiliary variables, not additional graph edges.

For integer $0\le t\le n$, define
\[
 M_v(t)=\sum_{i=1}^t I_v(i),\quad
 S(t)=\sum_{v\notin A_0}{\bf1}_{\{M_v(t)\ge R_v\}},\quad
 L(t)=a+S(t)-t.
\]
Before a nonseed vertex first becomes active it is unused, and all its
marks are genuine edges from distinct used vertices. After activation
its mark total stays above its threshold. Hence up to and including
real stopping the active count is exactly $a+S(t)$. Since
$L(t+1)-L(t)\ge-1$, exploration stops at
\begin{equation}\label{eq:lue-stop}
 T=\min\{0\le t\le n:L(t)\le0\},\qquad L(T)=0,
 \qquad |A_\infty|=T.
\end{equation}
The minimum exists because $L(n)\le0$. At stopping every active vertex
has been used, and every inactive vertex has fewer than $R_v$ neighbours
in that active set. The set is closed, and every activation used in
constructing it was a legal bootstrap activation, proving the last
identity. When $a=0$, $T=0$ directly. No activation after $T$ in the
auxiliary array is declared a real activation.

Independence of rows gives, at each deterministic integer time,
\begin{equation}\label{eq:lue-binomial}
 S(t)\sim\operatorname{Bin}(N,q(t)),\quad
 q(t)=1-(1-p)^t-(1-\pi)tp(1-p)^{t-1}\ (t\ge1),\quad q(0)=0.
\end{equation}
Different times need not be independent.
Write $F=F_{\rm E}$ from Section~\ref{sec:common}, $z=t/n$ and
$h_\alpha=\alpha+(1-\alpha)F-z$. Factorial moments give
\begin{equation}\label{eq:lue-local-use}
 q(t)\le z+d^2z^2/2,\qquad |q(t)-F(z)|\le d^2z/n.
\end{equation}
For the second bound, couple Bernoulli$(p)$ and Poisson$(p)$
with mismatch probability $p(1-e^{-p})\le p^2$, take $t$
independent copies and the same independent threshold.
For $\mu_a(t)=a+Nq(t)-t$ this gives
\begin{equation}\label{eq:lue-mean-error}
 |\mu_a(t)-nh_\alpha(t/n)|\le d^2z.
\end{equation}

\begin{lemma}[Maximal control of the mixed empirical process]
\label{lem:er-maximal}
Let $Y_1,\ldots,Y_N$ be independent identically distributed positive
integer variables with CDF $q$, where $q(0)=0$, $0<q(t)<1$ for
finite $t\ge1$, and $q(t)\uparrow1$. Let $N\le n$ and
$D_t=\sum_v\mathbf1_{\{Y_v\le t\}}-Nq(t)$. Then for $x>0$
and each integer $u\ge0$,
\begin{equation}\label{eq:er-maximal}
 \Pr(\sup_{t\le u}|D_t|>x)\le16nq(u)/x^2,\qquad
 \Pr(\sup_{t\ge0}|D_t|>x)\le8n/x^2.
\end{equation}
\end{lemma}
\begin{proof}
Use the own indicator filtrations
$\mathcal F_t=\sigma(\mathbf1_{\{Y_v\le s\}}:s\le t,v)$ and
$\mathcal G_t=\sigma(\mathbf1_{\{Y_v\le s\}}:s\ge t,v)$.
The processes $D_t/(1-q(t))$ and $D_t/q(t)$ are respectively
forward and reverse martingales where their denominators are nonzero.
For an indicator not yet one, the next hazard is
$(q(t+1)-q(t))/(1-q(t))$; in reverse, conditional on an indicator
being one at $t+1$, its probability of being one at $t$ is
$q(t)/q(t+1)$. These identities prove the conditional expectations;
independence permits summation. No thresholds are revealed in these
filtrations. Doob's $L^2$ inequality and
$\operatorname{Var}(D_u)=Nq(u)(1-q(u))$ give
\[
 \mathbb E\sup_{t\le u}|D_t|^2\le
       4Nq(u)/(1-q(u))\le8nq(u)\quad(q(u)\le1/2),
\]
and $\mathbb E\sup_{t\ge u}|D_t|^2\le
4N(1-q(u))/q(u)$ if $q(u)\ge1/2$.
For the reverse bound first use a finite reversed horizon, then let
it tend to infinity. Let $t_1$ be the last integer with $q(t_1)\le1/2$.
The left bound at $t_1$ and the right bound at $t_1+1$ are each
at most $4N$, even when $q$ jumps across $1/2$. Thus the global
second moment is at most $8n$. If $q(u)>1/2$, this is also at most
$16nq(u)$. Markov's inequality proves both assertions; $u=0$ is
deterministic. This is the empirical-process proof of
Lemmas~7.2--7.3 of \cite{janson2012}, applied to the mixed CDF.
\end{proof}

Apply this lemma to $Y_v=\min\{t:M_v(t)\ge R_v\}$.
Extend the iid auxiliary coins independently beyond $n$ if necessary;
these extra coins are not graph edges. The $Y_v$ are iid with the
mixed CDF in \eqref{eq:lue-binomial}, including its continuation.
Take reference sizes
\[
 a_- =\lfloor(1-\gamma/2)nA_{\rm E}\rfloor,\qquad
 a_+ =\lceil(1+\gamma/2)nA_{\rm E}\rceil.
\]
Choose $L$ as in Lemma~\ref{lem:barrier}.
Since $A^{\rm quad}/A_{\rm E}=1+O(d^{-1})$ and
$1/n=o(d^{-4})$, Lemma~\ref{lem:barrier} applies with margin
$\gamma/2$ and $o(A^{\rm quad})$ perturbations. Set
$\eta=\varepsilon$, $m=\lfloor n/2\rfloor$ and
$t_* =\operatorname{round}(nb/K)$. The three errors in
\eqref{eq:lue-mean-error}, relative to the corresponding queue
scales, are bounded by $Ld^3/n$, $d^3/(Ln)$ and $d^2/(2n)$,
all $o(1)$. Also $t_*=O(n/d^3)$ is early; $F'=O(1)$ there
makes time rounding an $O(1)=o(n/d^4)$ drift error.
Consequently, for some fixed $k>0$,
\begin{equation}\label{eq:er-mean-margins}
 \mu_{a_+}(t)\ge2k
 \begin{cases}
 n/d^4,&z\le L/d^3,\\
 nd^2z^2,&L/d^3<z\le\eta/d,\\
 n,&\eta/d<z\le1/2,
 \end{cases}
 \qquad \mu_{a_-}(t_*)\le-2kn/d^4.
\end{equation}

The early horizon $u=\lfloor nL/d^3\rfloor$ has
$q(u)=O(d^{-3})$. With budget $kn/d^4$, \eqref{eq:er-maximal}
gives failure $O(d^5/n)$, also controlling the negative test.
For middle intervals put $s_j=2^jL/d^3<\eta/d$ and
$u_j=\lfloor n\min(2s_j,\eta/d)\rfloor$.
For $s_j\le t/n\le\min(2s_j,\eta/d)$ use budget
$knd^2s_j^2$. Since $q(u_j)\le C(s_j+d^2s_j^2)$, failure is at most
\[
 C\{(nd^4s_j^3)^{-1}+(nd^2s_j^2)^{-1}\}.
\]
The sums of $2^{-3j}$ and $2^{-2j}$ are bounded, giving
$O((d^5+d^4)/n)$. These intervals and the early horizon cover
all integer times up to $\eta/d$; empty intervals may be skipped.
The late global bound with budget $kn$ costs $O(1/n)$.
No independence of intervals or survival conditioning is used. Hence
\begin{equation}\label{eq:er-passage}
 \Pr\{\exists t\le m:L_+(t)\le0\}\le Cd^5/n=o(1),
 \qquad \Pr\{L_-(t_*)\ge0\}\le Cd^5/n.
\end{equation}
By \eqref{eq:lue-stop}, the negative auxiliary test forces
$T\le t_*$, giving $|A_\infty|=O(n/d^3)$ below the threshold.
Above it, $T>m$, so the count at $m$ is a real active count.
Separately, unconditional binomial Hoeffding gives
$\Pr(|S(m)-Nq(m)|>n^{2/3})\le2e^{-2n^{1/3}}$.
Intersect the passage event $T>m$ with the event $|S(m)-Nq(m)|\le n^{2/3}$,
using the unconditional binomial bound. Since $p\le1/2$, $mp\le d/2$ and $mp\ge d/2-d/n$,
\[
 1-q(m)=(1-p)^m[1+(1-\pi)mp/(1-p)]
       \le(1+d)e^{-d/2+d/n}.
\]
Thus
\[
 n-|A_\infty|\le n-a-S(m)
       \le n(1+d)e^{-d/2+d/n}+n^{2/3},
\]
with failure at most $Cd^5/n+2e^{-2n^{1/3}}$.
This residual is $o(n)$ for every permitted $d\to\infty$.
Since $nA_{\rm E}\to\infty$, each allowed subcritical size is
at most $a_-$ and each allowed supercritical size is at least $a_+$
eventually. The coupling described in the paragraph on integer seed sizes in Section~\ref{par:seed-transfer} transfers the conclusions.
The estimates conditional on a seed set fixed in advance are uniform
in its labels. They do not permit adaptive seed selection or assert
a simultaneous guarantee over all sets. The cases $a=0,n$ are immediate.
This proves Theorem~\ref{thm:er}.

\section{Precision required of a static sandwich}\label{sec:sandwich}
For common thresholds and seeds, graph inclusion preserves active sets,
by induction on activation rounds. To examine the precision needed for
coefficient transfer, we first record a consequence of the fixed-parameter
theorem.

\begin{lemma}[A compact range of effective parameters]\label{lem:varying-c}
Suppose $r\to\infty$ and $r^5/n\to0$, and fix
$0<m_0\le M_0<\infty$ and $\gamma\in(0,1)$.
Let $c_n'\in[m_0,M_0]$ be deterministic and take $G(n,r/n)$ with
independent thresholds satisfying
$\Pr(R_i=1)=(1-c_n'/r)/r$ and independent uniform seeds of prescribed
size $a$. Put $T_n'=n(c_n')^2/(2r^4)$. Then
\[
 \begin{split}
 a\le(1-\gamma)T_n'&\ \Longrightarrow\quad
 |A_\infty|=O(n/r^3)\quad\hbox{with high probability},\\
 a\ge(1+\gamma)T_n'&\ \Longrightarrow\quad
 n-|A_\infty|=O(\mathcal R_{\rm E}(n,r))
 \quad\hbox{with high probability}.
 \end{split}
\]
The constants may depend on $m_0,M_0,\gamma$ only. The E conclusion
of Corollary~\ref{cor:logarithmic} also holds with $r$ in place of $d$.
\end{lemma}
\begin{proof}
Cover $[m_0,M_0]$ by finitely many intervals $[u_j,v_j]\subset(0,\infty)$
so short that
\[
 (1-\gamma)v_j^2\le(1-\gamma/2)u_j^2,\qquad
 (1+\gamma)u_j^2\ge(1+\gamma/2)v_j^2.
\]
Such a cover exists because both required upper bounds on $v_j/u_j$
are strictly larger than one. At fixed $r$, increasing the parameter
$c$ decreases the probability of threshold one. Couple thresholds using
common independent uniform labels. For fixed graph and seeds, the final
active set is then decreasing in $c$.
If $c_n'\in[u_j,v_j]$, the subcritical comparison is with the easier
model $c=u_j$, and the supercritical comparison with the harder model
$c=v_j$. The displayed inequalities supply margin $\gamma/2$.
Apply Theorem~\ref{thm:er} to the finitely many fixed endpoints and take
the maximum of their constants and a union bound over their errors.
If an interval is selected only on a subsequence, extend its tested seed
sizes outside that subsequence by zero for the subcritical application
and by $n$ for the supercritical application. The fixed-parameter theorem
then applies on a full sequence. Finally apply the explicit residual
bound to obtain the logarithmic assertion.
\end{proof}

Fix $c>0$ and let $d$ satisfy \eqref{eq:range-e}. For each $n$, keep
the threshold probability $\pi=(1-c/d)/d$ unchanged while replacing
the edge-probability parameter $d$ by $r=(1+\delta_n)d>0$.
Write $\beta=c/d$. Assume, for some
fixed $0<m\le M<\infty$,
\begin{equation}\label{eq:positive-deficit}
 \frac m d\le b_n':=1-\pi r
       =\beta-\delta_n(1-\beta)\le\frac M d.
\end{equation}
Then $\delta_n=O(1/d)$, $r/d\to1$, and $c_n'=rb_n'$ stays in a
fixed compact subinterval of $(0,\infty)$. Moreover
$\pi=(1-c_n'/r)/r$. Lemma~\ref{lem:varying-c} therefore gives the
leading seed scale $T_n'=n(b_n')^2/(2r^2)$, with both fixed-margin
outcomes and their stated residual bounds. Relative to the baseline
scale, the exact ratio of these leading expressions is
\begin{equation}\label{eq:sensitivity}
 \frac{T_n'}{nc^2/(2d^4)}
 =\frac{[1-\delta_n(1-\beta)/\beta]^2}{(1+\delta_n)^2}
 =\left(1-\frac{\delta_n}{\beta(1+\delta_n)}\right)^2.
\end{equation}
The expression before squaring is positive under
\eqref{eq:positive-deficit}. Consequently this ratio tends to one
if and only if $\delta_n=o(\beta)=o(1/d)$.
For example, $\delta_n=t/d$ with fixed $t<c$ gives the limiting ratio
$(c-t)^2/c^2$. A nonpositive or vanishing normalized deficit is outside
\eqref{eq:positive-deficit}; squaring does not extend the theorem there.
Suppose every displacement in the band $|\delta|\le\varepsilon_n$
satisfies \eqref{eq:positive-deficit} with the same fixed constants
$m,M$. The leading coefficient is then preserved to relative error
$o(1)$ uniformly over the band if and only if $\varepsilon_n=o(1/d)$.
Necessity follows already from the endpoint
$\delta=-\varepsilon_n$, and sufficiency follows from
\eqref{eq:sensitivity}. This is a statement about a full band, not
about a possibly much smaller actual displacement within an upper bound.

\begin{proposition}[Static inclusion at coefficient precision]
\label{prop:static-sandwich}
Let $d=d_n$ be integer with $nd$ even, $d\to\infty$ and
$d=o(n^{1/7})$, and let $0\le\varepsilon_n=o(1/d)$.
For every coupling on the same labelled vertex set of
\[
 G^-\sim G(n,(1-\varepsilon_n)d/n),\qquad
 G_{\rm R}\sim G_{n,d},
\]
where $G_{n,d}$ is the uniform simple $d$-regular graph,
\[
 \Pr(G^-\subseteq G_{\rm R})\le\tfrac12+o(1).
\]
The error term can be chosen independently of the coupling.
In particular, inclusion cannot hold with probability tending to one.
\end{proposition}
\begin{proof}
At any fixed vertex let $Y$ be its degree in $G^-$. Set
$p=(1-\varepsilon_n)d/n$ and $N=n-1$. Then
\[
 Y\sim\operatorname{Bin}(N,p),\quad
 \mu=Np=(1-\varepsilon_n)d(1-1/n),\quad
 \sigma^2=Np(1-p)\sim d\longrightarrow\infty,
\]
\[
 d-\mu=\varepsilon_n d+(1-\varepsilon_n)d/n=o(1).
\]
For completeness, write $Y-\mu=\sum_{j=1}^N Z_j$ with independent
centred Bernoulli summands. For fixed real $t$,
\[
 \mathbb E e^{itZ_j/\sigma}
 =1-\frac{t^2\mathbb E Z_j^2}{2\sigma^2}+\rho_j,
 \qquad
 \sum_j|\rho_j|\le
 \frac{|t|^3}{6\sigma^3}\sum_j\mathbb E|Z_j|^3
 \le\frac{|t|^3}{6\sigma}\longrightarrow0.
\]
The maximum modulus of the terms added to one tends to zero, and
the sum of their moduli is bounded. Taking logarithms of their product
therefore yields $-t^2/2+o(1)$. Thus $(Y-\mu)/\sigma$ converges in
law to a standard normal variable. Since $(d-\mu)/\sigma\to0$,
$\Pr(Y\le d)\to1/2$.
Inclusion in a $d$-regular graph requires $Y\le d$, whatever the
coupling. Its probability is therefore bounded by this same marginal
probability, proving the assertion.
\end{proof}

Together with \eqref{eq:sensitivity}, this obstructs direct transfer
of the leading coefficient by a static two-sided inclusion sandwich
of relative width $o(1/d)$ in the common degree range.
It does not rule out coarser sandwiches or other comparison arguments.

\section{Deterministic finite-degree corrections}\label{sec:finite-degree}
We now refine the quadratic approximation $A=b^2/(2K)$ using the
deterministic response functions. The assumptions are stated locally:
these calculations do not require the degree ranges used in the
probability theorems.

\begin{equation}\label{eq:det-barrier}
 \alpha_{\rm det}=\max_{0\le z\le1/2}R(z),\qquad
 R(z)=\frac{z-F(z)}{1-F(z)}.
\end{equation}
Indeed $h_\alpha(z)=(1-F(z))(\alpha-R(z))$, and $1-F(z)>0$
on this interval. This barrier refines the deterministic quadratic
approximation; it is not an identification of the finite-$n$ random
critical-window centre.

\subsection{Exact regular barrier}
For regular degrees, let $d\ge4$ and $0\le\pi<1/(d-1)$, and put
$b=1-\pi(d-1)\in(0,1]$. The deterministic barrier is explicit:
\[
 \widehat z_{\rm reg}=\frac{b}{(d-2)(d-3+b)},\qquad
 \alpha_{\rm det}^{\rm reg}
 =1-\frac{(1+b/(d-3))^{d-3}}{(1+b/(d-2))^{d-2}}.
\]
Indeed $R_d=1-J_d^{-1}$, where
$J_d=(1-z)^{d-3}[1+(d-3+b)z]$, and
$J_d'=(1-z)^{d-4}[b-(d-2)(d-3+b)z]$.
The optimiser satisfies $0<\widehat z_{\rm reg}\le1/(d-2)^2\le1/4$, so it belongs to the interval in \eqref{eq:det-barrier}.
At $\pi=0$ this reduces to the threshold-two response formula of
\citet{balogh}; here it is used only as a deterministic identity.
For the choice $\pi=(1-c/d)/d$, put $m=d-3$. The logarithm of
$J_d(\widehat z_{\rm reg})$ is
\[
 \ell=(m+1)\log(1+b/(m+1))-m\log(1+b/m)
       =\frac{b^2}{2m(m+1)}+O(d^{-6}).
\]
The linear terms cancel; the cubic difference is
$\frac{b^3}{3}[(m+1)^{-2}-m^{-2}]=O(d^{-6})$,
and each fourth-order remainder is $O(d^{-7})$.
Since $\alpha_{\rm det}^{\rm reg}=1-e^{-\ell}=\ell+O(d^{-8})$,
$b=(c+1)/d-c/d^2$ and $[m(m+1)]^{-1}=d^{-2}(1+5/d+O(d^{-2}))$ give
\[
 \alpha_{\rm det}^{\rm reg}
 =\frac{(c+1)^2}{2d^4}
  +\frac{(c+1)(3c+5)}{2d^5}+O(d^{-6}).
\]

\subsection{General degree sequence}
\begin{proposition}[Finite-degree expansion]\label{prop:finite-degree}
Fix $c>0$, $C\ge0$ and $\kappa>0$. Let $D$ have the empirical law of
nonnegative integers $d_1,\ldots,d_n$, with exact mean $d\to\infty$, and assume
\[
 \max_i|d_i-d|\le C\sqrt d,\qquad
 v_n=\operatorname{Var}(D)/d,\qquad \rho=c+1-v_n\ge\kappa.
\]
Set $\pi=d^{-1}-cd^{-2}$ and $(x)_j=x(x-1)\cdots(x-j+1)$. For sufficiently
large $d$ define
\[
 \begin{split}
 B(z)&=\frac1d\mathbb{E}\!\left[D\left\{(1-z)^{D-1}
       +(1-\pi)(D-1)z(1-z)^{D-2}\right\}\right],\qquad F=1-B,\\
 R(z)&=\frac{z-F(z)}{1-F(z)},\qquad
 \alpha_{\rm det}=\max_{0\le z\le1/2}R(z),\\
 b&=1-\pi\frac{\mathbb{E}[(D)_2]}d,\qquad
 K=(1-2\pi)\frac{\mathbb{E}[(D)_3]}d,\qquad A=\frac{b^2}{2K},\\
 L_3&=-(2-3\pi)\frac{\mathbb{E}[(D)_4]}d,\qquad
 L_4=(3-4\pi)\frac{\mathbb{E}[(D)_5]}d.
 \end{split}
\]
The maximum is attained uniquely, at
\begin{equation}\label{eq:det-location}
 \widehat z=\frac bK-\frac{L_3b^2}{2K^3}
                  +\frac{b^2}{2K^2}+O(d^{-7}),
\end{equation}
and
\begin{equation}\label{eq:det-expansion}
\begin{split}
 \alpha_{\rm det}={}&A-\frac{L_3b^3}{6K^3}+\frac{Ab}{K}-A^2\\
 &+\frac{L_3^2b^4}{8K^5}-\frac{L_4b^4}{24K^4}+O(d^{-9}).
\end{split}
\end{equation}
All coefficients here are exact empirical moments. The remainders are
uniform for fixed $c,C,\kappa$; no limit of $v_n$ and no relation between
$n$ and $d$ are required. This is a deterministic response statement,
not a graphicality assertion or a random critical-window theorem.
\end{proposition}
\begin{proof}
The deterministic P response verification in Section~\ref{sec:common}
uses exactly the assumptions of this proposition. It gives
$F'(0)=1-b$, $F''(0)=K$, $F'''(0)=L_3$, $F''''(0)=L_4$
and $|F^{(j)}|=O(d^j)$ for $2\le j\le5$ on the early interval.
Lemma~\ref{lem:barrier} gives the unique global maximiser in
$(0,8b/K)$, with $b\asymp d^{-1}$, $K\asymp d^2$ and
$A\asymp d^{-4}$. No probability degree range is used.

Put $z_0=b/K$ and $N=z-F$. On $[0,8b/K]$,
\[
 N(z)=bz-Kz^2/2-L_3z^3/6-L_4z^4/24+O(d^5z^5),\qquad
 R''(z)=-K(1+O(d^{-2})).
\]
Here $z_0=O(d^{-3})$, $L_3=O(d^3)$ and $L_4=O(d^4)$.
At $z_0$ the numerator of $R'$ is
\[
 N'(z_0)B(z_0)+N(z_0)F'(z_0)=-L_3z_0^2/2+A+O(d^{-5}).
\]
The mean value theorem first gives $\widehat z-z_0=O(d^{-5})$;
the same identity and curvature then give \eqref{eq:det-location}.
Writing $\delta_3=-L_3z_0^2/(2K)$, we have
$\widehat z-z_0=\delta_3+O(d^{-6})$. Completing the square yields
\[
 N(\widehat z)=A-L_3z_0^3/6-L_4z_0^4/24
                    +L_3^2z_0^4/(8K)+O(d^{-10}).
\]
Indeed the neglected terms are bounded by constant multiples of
$d^5z_0^5$, $Kd^{-12}$, $|L_3|z_0\delta_3^2$ and
$|L_4|z_0^3|\delta_3|$, each $O(d^{-10})$.
Finally $F(\widehat z)=z_0-A+O(d^{-5})$, whence
\[
 R(\widehat z)=N(\widehat z)(1+F(\widehat z)+O(F(\widehat z)^2))
             =N(\widehat z)+Az_0-A^2+O(d^{-9}).
\]
The denominator thus leaves an $O(d^{-9})$ remainder, although the
numerator remainder above is $O(d^{-10})$. This proves the claim.
\end{proof}

The expansion separates three levels of degree information.
The leading term depends on $v_n$ through the excess-degree deficit.
The third central moment can first contribute at absolute order
$d^{-11/2}$. After matching the first three moments, a fourth-moment
difference of order $d^2$ can still change the deterministic barrier
at order $d^{-8}$. We now give the coefficients and an explicit example.

\subsection{Variance, skewness, and two degree values}
Let $\mu_j=\E[(D-d)^j]$, $\rho=c+1-v_n$, $r=c(v_n-1)$,
$k_1=3v_n-5$ and $k_2=8-9v_n+2c$. The exact identities
\[
 b=\rho/d+r/d^2,\qquad
 K=\left(1-\frac2d+\frac{2c}{d^2}\right)
   \left[d^2+(3v_n-3)d+2-3v_n+\frac{\mu_3}d\right]
\]
and $\mu_3=O(d^{3/2})$ give, by expanding \eqref{eq:det-expansion},
\begin{equation}\label{eq:moment-expansion}
\begin{split}
 \alpha_{\rm det}=\frac{\rho^2}{2d^4}\bigg[
  1&+\frac{2r/\rho-k_1}{d}-\frac{\mu_3}{d^3}\\
   &+\frac{r^2/\rho^2+k_1^2-k_2-2rk_1/\rho+2\rho/3}{d^2}
    +O(d^{-5/2})\bigg].
\end{split}
\end{equation}
The term $2\rho/3$ comes from the cubic response; the denominator correction
$Ab/K$ has absolute order $d^{-7}$ and does not enter the displayed
relative $d^{-2}$ coefficient. If $\mu_3=0$, the remainder inside
brackets improves to $O(d^{-3})$.
For equal empirical masses at $d-s$ and $d+s$, require even $n$ and
nonnegative integer values $d-s,d+s$; then $v_n=s^2/d$ and $\mu_3=0$ exactly.
Graph applications also require even total degree $nd$ and the assumptions
of the prescribed-degree probability theorem. The absolute remainder in
\eqref{eq:moment-expansion} is $O(d^{-13/2})$, or $O(d^{-7})$ when $\mu_3=0$;
it is not the $O(d^{-9})$ remainder of the exact-coefficient formula.
These expansions use actual empirical moments, with $c+1-v_n\ge\kappa>0$.
A variance limit is optional and cannot replace $v_n$ in the correction
terms without a rate. For example, for integer $d$ and fixed $0\le v<c+1$,
$s=\lfloor\sqrt{vd}\rfloor$ gives $v_n=v+O(d^{-1/2})$; this error can exceed
the first relative $d^{-1}$ correction.

\subsection{A distinction beyond variance}
The third central moment can affect the barrier at order $d^{-11/2}$
through $K$, as \eqref{eq:moment-expansion} shows.
Even matching that moment leaves a further distinction.
Indeed, expansion of the fourth factorial moment gives
\[
 \frac{\E[(D)_4]}d
 =d^3+(6v_n-6)d^2+(11-18v_n)d+(11v_n-6)
 +(4-6/d)\mu_3+\frac{\mu_4}d.
\]
For two degree laws with the same $d,v_n,\mu_3$ and a common deviation
bound, $b,K$ agree exactly. Their $L_3$ difference is
$-(2-3\pi)(\mu_{4,2}-\mu_{4,1})/d=O(d)$, whereas their $L_4$
difference is $O(d^2)$: this follows by expanding $(D)_5$ and using
$\mu_4=O(d^2)$, $\mu_5=O(d^{5/2})$.
Consequently the differences of the $L_3^2$ and $L_4$ terms in
\eqref{eq:det-expansion} are $O(d^{-10})$, and
\begin{equation}\label{eq:kurtosis-difference}
 \alpha_{{\rm det},2}-\alpha_{{\rm det},1}
 =\frac{(2-3\pi)b^3}{6dK^3}(\mu_{4,2}-\mu_{4,1})
   +O(d^{-9}).
\end{equation}
For a concrete example take an integer $s\to\infty$ and put $d=s^2$.
The law with mass $1/2$
at each of $d-s,d+s$, and the law with mass $3/4$ at $d$ and $1/8$
at each of $d-2s,d+2s$, both have $v_n=1$ and $\mu_3=0$.
Their fourth central moments differ by $3d^2$, so
\[
 \alpha_{{\rm det},2}-\alpha_{{\rm det},1}
       =c^3d^{-8}+O(d^{-9}).
\]
Both are empirical degree laws when $n$ is divisible by eight, and
admissible sequences can be chosen in \eqref{eq:degrees}.
This distinguishes deterministic barriers at equal mean and variance;
it does not resolve random threshold differences at this scale.

\subsection{Independent-edge response}
Fix $c>0$ and set $\pi=d^{-1}-cd^{-2}$. The Poisson response and its
deterministic barrier are
\[
 Q(x)=1-e^{-x}[1+(1-\pi)x],\qquad
 P(z)=\frac{z-Q(dz)}{1-Q(dz)},\qquad
 \alpha_{{\rm E},{\rm Pois}}=\max_{0\le z\le1/2}P(z).
\]
For $d\to\infty$, the Taylor calculation in
Proposition~\ref{prop:finite-degree}, applied directly to $F_{\rm E}(z)=Q(dz)$,
uses the exact coefficients
\[
 b_{\rm E}=c/d,\quad K_{\rm E}=(1-2\pi)d^2,\quad
 L_{3,\rm E}=-(2-3\pi)d^3,\quad L_{4,\rm E}=(3-4\pi)d^4.
\]
It gives \eqref{eq:det-expansion} with these coefficients and
$A_{\rm E}^{\rm quad}=b_{\rm E}^2/(2K_{\rm E})$, and in particular
\begin{equation}\label{eq:fx-er-poisson}
 \alpha_{{\rm E},{\rm Pois}}
 =\frac{c^2}{2d^4}\left[1+\frac2d+\frac{4-4c/3}{d^2}\right]+O(d^{-7}).
\end{equation}
The derivative and localization bounds needed for this calculation
are verified in the proof below; no bounded-support degree assumption
is imposed on the Poisson response.

\begin{proposition}[Integer-time binomial barrier]\label{prop:fx-er-bridge}
Let $n$ be an integer, $d\to\infty$ and $n/d^3\to\infty$. Put $p=d/n$ and,
for integers $0\le t\le\lfloor n/2\rfloor$, define
\[
 q(t)=1-(1-p)^t\left[1+\frac{(1-\pi)tp}{1-p}\right],\qquad
 \beta_n=\max_{0\le t\le\lfloor n/2\rfloor}
                 \frac{t/n-q(t)}{1-q(t)}.
\]
The maximum is over integers, and $q(0)=0$. Then
\begin{equation}\label{eq:fx-er-bridge}
 |\beta_n-\alpha_{{\rm E},{\rm Pois}}|
       =O\left(\frac1{nd}+\frac{d^2}{n^2}\right),
\end{equation}
and
\begin{equation}\label{eq:fx-er-binomial}
 \beta_n=\frac{c^2}{2d^4}\left[1+\frac2d+\frac{4-4c/3}{d^2}\right]
       +O\left(d^{-7}+\frac1{nd}+\frac{d^2}{n^2}\right).
\end{equation}
The constants in \eqref{eq:fx-er-bridge} and \eqref{eq:fx-er-binomial}
depend only on $c$. Under $d^5/n\to0$ the last
remainder is $o(d^{-6})$. An $O(d^{-7})$ remainder is obtained if additionally
$d^6=O(n)$. The full exact-coefficient expression \eqref{eq:det-expansion},
with the Poisson coefficients above, transfers to $\beta_n$ with remainder
$O(d^{-9})$ if $d^8=O(n)$. Constants in these simplified remainders may
also depend on the bounds for $d^6/n$ or $d^8/n$, respectively.
\end{proposition}
\begin{proof}
All assertions are for sufficiently large $d$, so $0<\pi<1$ and $p<1/2$.
For $x=tp$, $r=t[-\log(1-p)-p]$ and $A_x=1+(1-\pi)x$,
\[
 1-q(t)=e^{-x-r}\left[A_x+\frac{(1-\pi)xp}{1-p}\right],\qquad
 0\le r\le\frac{xp}{2(1-p)}.
\]
Using $1-e^{-r}\le r$ and boundedness of $e^{-x}(1+x)$ on $x\ge0$ gives,
with $z=t/n$,
\begin{equation}\label{eq:fx-er-response-error}
 |q(t)-Q(x)|\le e^{-x}\left[A_xr+\frac{(1-\pi)xp}{1-p}\right]
              \le C_1xp=C_1d^2z/n.
\end{equation}

By the response verification and Lemma~\ref{lem:barrier}, the Poisson
barrier has a unique maximiser $z_P\in(0,8b_{\rm E}/K_{\rm E})$.
For a fixed $L>8b^*/K_*$, its response satisfies
$F_{\rm E}-z\ge k d^2z^2$ on $[L/d^3,\varepsilon/d]$ and
$F_{\rm E}-z\ge k$ on $[\varepsilon/d,1/2]$.
These are bounds for the numerator, before any denominator comparison.
The response-error ratios in \eqref{eq:fx-er-response-error} are
at most $C/(nz)\le Cd^3/(Ln)=o(1)$ in the middle interval and
$O(d^2/n)=o(1)$ in the late interval, using only $n/d^3\to\infty$.
Thus $q(t)>t/n$ at every grid point outside $[0,L/d^3]$;
their binomial barrier values are negative since $1-q(t)>0$.

On $[0,L/d^3]$,
$|F_{\rm E}^{(j)}|=O(d^j)$ for $2\le j\le5$, $F_{\rm E}'=O(1)$,
and $P''=-K_{\rm E}(1+O(d^{-2}))$. The Taylor calculation already proved
therefore gives $z_P\asymp d^{-3}$,
$\alpha_{{\rm E},{\rm Pois}}\asymp d^{-4}$ and the full
exact-coefficient expansion. Expanding
\[
 A_{\rm E}^{\rm quad}
 =\frac{c^2}{2d^4}\left[1+\frac2d+\frac{4-2c}{d^2}+O(d^{-3})\right]
\]
and adding $-L_{3,\rm E}b_{\rm E}^3/(6K_{\rm E}^3)
=c^3/(3d^6)+O(d^{-7})$ proves \eqref{eq:fx-er-poisson}; the other terms
start at order $d^{-7}$.

Only on the localized interval do we compare denominators. There
$F_{\rm E}=O(d^{-3})$ and, by \eqref{eq:fx-er-response-error},
$q=O(d^{-3})$ at grid points, so both denominators are $1-O(d^{-3})$.
The exact identity
\[
 \frac{z-q(t)}{1-q(t)}-P(z)
 =-\frac{(1-z)(q(t)-F_{\rm E}(z))}{(1-q(t))(1-F_{\rm E}(z))}
\]
then bounds the difference by $C_3/(nd)$. Choose an integer $t_P$ nearest
to $nz_P$. Since $n/d^3\to\infty$, its grid point is interior to the early
interval and to $[0,\lfloor n/2\rfloor/n]$. Stationarity $P'(z_P)=0$
and $|P''|=O(d^2)$ give
\[
 P(t_P/n)=\alpha_{{\rm E},{\rm Pois}}+O(d^2/n^2).
\]
Both $1/(nd)$ and $d^2/n^2$ are $o(d^{-4})$. Thus this grid point has
positive binomial barrier value; the binomial maximum is early.
Comparison at this point gives the lower bound in
\eqref{eq:fx-er-bridge}, and comparison at a binomial maximiser gives its
upper bound. No uniqueness of the integer maximiser is required.

Combining with \eqref{eq:fx-er-poisson} proves
\eqref{eq:fx-er-binomial}. Under $d^5/n\to0$,
$d^8/n^2=(d^5/n)^2/d^2\to0$, proving the $o(d^{-6})$ assertion.
If $d^6=O(n)$, then $(nd)^{-1}=O(d^{-7})$ and $d^2/n^2=O(d^{-10})$.
If $d^8=O(n)$, these errors are respectively $O(d^{-9})$ and $O(d^{-14})$,
which proves the full-expansion assertion.
\end{proof}

The integer-time error $O(d^2/n^2)$ is distinct from rounding the number
of seeds. For an integer seed size $a$ and $\alpha=a/n$, the exact mean drift is
\[
 \frac{a+(n-a)q(t)-t}{n}
 =(1-q(t))\left(\alpha-\frac{t/n-q(t)}{1-q(t)}\right).
\]
Consequently the smallest integer $a$ making all these mean drifts
nonnegative is $\lceil n\beta_n\rceil$; for strictly positive mean drifts
it is $\lfloor n\beta_n\rfloor+1$. Their density errors are respectively
less than $1/n$ and at most $1/n$. This additional seed-rounding error is
not part of \eqref{eq:fx-er-bridge} or \eqref{eq:fx-er-binomial}.
All these barriers concern exact deterministic means on the stated time
interval. They do not locate the random critical-window centre or
justify shrinking relative margins in the probability theorems.

\section{Scope}\label{sec:scope}
The results concern independent uniform seeds of prescribed size,
fixed positive model parameters and fixed relative margins.
The general prescribed-degree and independent-edge theorems establish
almost complete activation. For regular graphs,
Corollary~\ref{cor:regular} gives full activation in the same degree range,
without a logarithmic lower bound on $d$.
The results do not identify optimal seeds or a shrinking critical window.
The prescribed assumptions impose bounded degree deviations and a
fixed positive excess-degree gap; they do not cover arbitrary
heavy-tailed degrees or thresholds correlated with the graph.

For prescribed degrees the local concentration exponent $n/d^5$ must
dominate the $O(d^2)$ simplicity cost and the logarithm of the grid
factor; $d=o(n^{1/7})$ also controls the separate late estimate.
For independent edges the maximal empirical-process estimate and dyadic
budgets give passage failure $O(d^5/n)$, so $d^5/n\to0$ suffices.
These are sufficient ranges, not optimal endpoints.

The explicit residuals tend to zero as fractions of $n$ for every
allowed sequence with $d\to\infty$.

Corollary~\ref{cor:logarithmic} gives sufficient logarithmic conditions
for an $O(n^{2/3})$ inactive remainder in the general models.
They are not necessity claims:
Appendix~\ref{app:rossi-overlap} also gives this remainder in the
prescribed slow-degree regime $d\log d=o(\log n)$.
That adaptation does not establish the independent-edge counterpart
or complete the subcritical terminal comparison with Rossi et al.

The finite-degree results are deterministic response statements.
For the integer-time binomial barrier the three-term expansion has
absolute remainder $O(d^{-7}+(nd)^{-1}+d^2/n^2)$, which is
$o(d^{-6})$ throughout the independent-edge theorem's range.
Rounding the number of seeds adds a separate density error $O(1/n)$.
Neither error is a stochastic window estimate.

If $c+1-v_n$ approaches zero, the exact identity in
\eqref{eq:scale} shows that the order and sign of $b$ require a
separate analysis. No boundary-gap conclusion is asserted here.
An empty seed set remains empty because all thresholds are positive.

Appendix~\ref{app:bounded-tail} gives fixed-margin terminal bounds for iid
thresholds independent of degree, with fixed support bound $r_0$ and
$q_{2,n}\ge q_*>0$. Growing support, vanishing threshold-two mass,
degree-dependent thresholds and a vanishing deficit lie outside its hypotheses.
The extension gives neither finite-degree corrections nor quenched laws
and does not identify a shrinking critical window or its limiting distribution.

\appendix
\section{Local concentration for the prescribed exploration}
\label{lup:appendix}
We prove Lemma~\ref{lem:lup-local}, retaining the fixed-size seed law. For independent centred variables of absolute value at most $B$ and total
variance at most $V$, we use the two-sided bound
$2\exp[-x^2/(2(V+Bx/3))]$. For a martingale starting at zero, the same
bound applies to its running absolute maximum if its jumps are bounded
by $B$ and its predictable quadratic variation is at most $V$.
These follow from
$e^t-1-t\le t^2/[2(1-|t|/3)]$ for $|t|<3$, the associated exponential
(super)martingale, and the choice $\lambda=x/(V+Bx/3)$; applying both
signs gives the factor two. The martingale bound is applied only to
stopped processes with a deterministic quadratic-variation bound.

\paragraph{One-vertex loss and the seed mean.}
At fixed $u$, put $z=1-e^{-u}$ and let $X\sim\operatorname{Bin}(k,z)$
be independent of $R$. For $Y=(k-X)\mathbf1_{\{X<R\}}$, write
\[
 L=k-Y=k\mathbf1_{\{X\ge R\}}+\mathbf1_{\{R=2,X=1\}}.
\]
The two events are disjoint: activation loses all $k$ ends, and a
threshold-two vertex with one mark loses one. Hence
\[
 \operatorname{Var}Y\le\mathbb EL^2
 =k^2[\pi\Pr(X\ge1)+(1-\pi)\Pr(X\ge2)]
   +(1-\pi)\Pr(X=1)
 \le\pi k^3z+kz+k^4z^2/2,
\]
using $\Pr(X\ge1),\Pr(X=1)\le kz$ and
$\Pr(X\ge2)\le k^2z^2/2$.
Conditionally on a fixed seed set, different vertex summands are
independent and lie in $[0,\Delta]$. Thus
\begin{equation}\label{lup:variance}
 \operatorname{Var}(H(u)\mid A_0)\le V(z):=
 \min\{3n[(\pi\Delta^3+\Delta)z+\Delta^4z^2/2],\ n\Delta^2\}.
\end{equation}
Let $g_i(u)=\mathbb EY_i(u)\in[0,\Delta]$ and
$m_{A_0}(u)=\sum_{i\notin A_0}g_i(u)$.
Reveal the $a$ seed labels in random order without replacement.
For $a<n$ the Doob increment of the sampled sum at draw $j$ is
\[
 \frac{n-a}{n-j}\bigl(g_{I_j}(u)-\bar g_{j-1}(u)\bigr),
 \qquad 1\le j\le a,
\]
where $\bar g_{j-1}$ averages the labels not yet drawn. This follows
by writing the expected remaining sample as $a-j$ times its remaining
population mean. Its absolute value is at most $\Delta$, so the
bounded-increment exponential estimate gives
\begin{equation}\label{lup:seed}
 \Pr\{|m_{A_0}(u)-\bar H(u)|>e\}
 \le2e^{-e^2/(2a\Delta^2)}.
\end{equation}
For $a=0,n$ the discrepancy is zero. The same estimate includes the
variation in seed degrees at $u=0$. For the reference sizes,
$a\Delta^2=O(n/d^2)$.

\paragraph{Stopped queue martingale.}
Use the exploration filtration, which does not reveal future clocks
of manually removed ends. If $J$ counts pairings through $\tau$, its
compensator is $\int_0^{u\wedge\tau}W(s-)\,ds$: before stopping
there are $W$ unreserved rate-one clocks. The fictitious terminal
reservation makes every jump of $W$ exactly $-2$, also at $\tau$.
Consequently
\[
 Z(u)=e^{2(u\wedge\tau)}W(u\wedge\tau)-(M-1)
 =-2\int_0^{u\wedge\tau}e^{2s}(dJ(s)-W(s-)\,ds)
\]
is a martingale. For each deterministic $r\le U$,
\[
 |\Delta Z|\le2e^{2r},\qquad
 \langle Z\rangle_r
 =4\int_0^{r\wedge\tau}e^{4s}W(s-)\,ds\le4e^{4r}Mr.
\]
The value $W(\tau)$ may be $-1$ if the pairing is exhausted, but the
intensity uses the nonnegative preterminal value $W(s-)$.
The martingale inequality therefore gives
\begin{equation}\label{lup:W}
 \Pr\left\{\sup_{u\le r\wedge\tau}|W(u)-Me^{-2u}|>e+1\right\}
 \le p_W(e,r):=2\exp\left[-\frac{e^2}
 {2(4e^{4r}Mr+2e^{2r}e/3)}\right].
\end{equation}
The extra one accounts for $W(0)=M-1$.

\paragraph{Time interpolation.}
For any deterministic interval $I=[l,r]\subset[0,U]$ and $e\ge1$,
use a grid containing both endpoints with mesh at most
$e/(4n\Delta^2)$. Its size is at most
$N_I=\lceil4(r-l)n\Delta^2/e\rceil+1$ (one point for a degenerate
interval). At a grid point, conditional Bernstein and \eqref{lup:seed}
give $|H-\bar H|\le2e$, except with probability at most
\[
 2\exp\left[-\frac{e^2}{2(V(1-e^{-r})+\Delta e/3)}\right]
 +2\exp[-e^2/(2a\Delta^2)].
\]
Each $Y_i(u)$ is nonincreasing. Its clock-event rate is at most $d_i$
and each loss at most $d_i$, whence $|g_i'(u)|\le d_i^2$.
Thus both $H$ and $\bar H$ are nonincreasing and
$|\bar H'|\le n\Delta^2$. Bracketing an intermediate value of $H$
by its two adjacent grid values, and varying the mean by at most
$e/4$, gives $|H-\bar H|\le9e/4$ on all of $I$, including jumps.
Combining with \eqref{lup:W} yields
\begin{equation}\label{lup:interval}
 |Q-D_\alpha|\le4e+1\quad\hbox{on }I\cap[0,\tau],
\end{equation}
except on an event of probability at most
\[
 p_W(e,r)+N_I\left\{
 2\exp\left[-\frac{e^2}{2(V(1-e^{-r})+\Delta e/3)}\right]
 +2\exp[-e^2/(2a\Delta^2)]\right\}.
\]
This is a union bound, with no independence between intervals, between
$W$ and $H$, or between the estimates and survival.

\paragraph{Three scales and the total passage probability.}
Partition the $z$ interval into $[0,L/d^3]$, dyadic intervals
$[s,\min(2s,\eta/d)]$ starting at $s=L/d^3$, and $[\eta/d,1/2]$.
Map them to $u=-\log(1-z)$ and choose budgets respectively
\[
 e_0=\rho n/d^3,\qquad e_s=\rho nd^3s^2,\qquad e_{\rm p}=\rho nd.
\]
Fix $\rho>0$ sufficiently small relative to \eqref{lup:margins}.
The minimum budget tends to infinity, so $4e+1$ is smaller than half
the appropriate margin, including the negative test at $u_*$.
On the early interval, \eqref{lup:variance} is $O(n/d)$,
$a\Delta^2=O(n/d^2)$ and $Mr=O(n/d^2)$.
The three exponents (vertex Bernstein, seed, queue) are therefore
bounded below by constant multiples of $n/d^5,n/d^4,n/d^4$.

On a middle interval $r=O(s)$ and $z_r\le2s$. The corresponding
exponents have the lower bounds
\[
 k\frac{nd^4s^3}{1+d^2s},\qquad
 k nd^8s^4,\qquad
 k\frac{nd^5s^3}{1+d^2s},
\]
respectively. Here $V(z_r)+\Delta e_s\le Cn(d^2s+d^4s^2)$
and $Mr+e_s\le Cnd s(1+d^2s)$.
Since $s^3/(1+d^2s)$ is increasing and $s\ge L/d^3$, these bounds
are at least constant multiples of $n/d^5,n/d^4,n/d^4$.
On the last interval use $V\le n\Delta^2$: all three exponents
are at least $kn$. This last budget is for passage only.

The early grid has $O(1+d^2)$ points; a middle grid has
$O(1+1/(ds))=O(1+d^2)$ points; the late grid has $O(1+d)$ points.
There are $O(1+\log d)$ intervals, because the ratio of the middle
endpoints has order $d^2$. Summing \eqref{lup:interval}'s failure
bounds proves \eqref{lup:localfail}. This completes the local
estimate without a global error of order $n^{2/3}d$ at the bottleneck.

\section{The slow-degree overlap with finite-time approximation}
\label{app:rossi-overlap}

We record the adaptation used in the literature comparison, without
introducing a separate threshold theorem. Fix $c>0$, $C\ge0$,
$\kappa>0$ and $\gamma\in(0,1)$. The prescribed nonnegative integer
degrees have even sum $nd$, exact mean $d$, and satisfy
\[
 |d_i-d|\le C\sqrt d,\qquad
 v_n=\frac{1}{nd}\sum_i(d_i-d)^2,\qquad c+1-v_n\ge\kappa,
 \qquad d\to\infty,\quad d\log d=o(\log n).
\]
Let $p_k=n^{-1}|\{i:d_i=k\}|$, $\delta=\min_i d_i$ and
$\Delta=\max_i d_i$. Independently of the graph, thresholds are
independent with probabilities $\pi=(1-c/d)/d$ and $1-\pi$ for
one and two, and the seed is a uniform subset of deterministic size $a$.
With expectation taken under $p_k$, set
\[
 b=1-\pi\frac{\mathbb E[D(D-1)]}{d},\qquad
 K=(1-2\pi)\frac{\mathbb E[D(D-1)(D-2)]}{d},\qquad A=\frac{b^2}{2K}.
\]
Then $b\asymp d^{-1}$, $K\asymp d^2$, and $A\asymp d^{-4}$,
uniformly for the fixed parameters above. We show that
$a\ge(1+\gamma)nA$ implies $n-|A_\infty|=O(n^{2/3})$ with
probability tending to one on the uniform simple graph. The probability
includes the graph, thresholds and seed. Since $nA\to\infty$, it suffices
by the paragraph on integer seed sizes in Section~\ref{par:seed-transfer} to use
$a_0=\lceil(1+\gamma/2)nA\rceil$ and $\alpha=a_0/n$.

On undirected trees the nonroot recursion excludes the parent edge,
as in Lemma~4 of the accepted manuscript of~\cite{rossi}.

The printed Theorem~2 refers to recursion~(7), whose formulas
(13)--(14) use $k$ rather than $k-1$ trials. We therefore prove the
finite-time estimate needed for the excess-degree recursion directly,
following the coupling and concentration strategy of Section~4.2
of~\cite{rossi}.

\begin{lemma}[Finite-time approximation for an undirected pairing]
\label{lem:undirected-finite-time}
Fix degrees $d_i\ge0$ with even sum $M=nd>0$ and maximum
$\Delta\ge2$, and fix integer thresholds $r_i$ and labels $s_i$, where
$s_i\in\{0,1\}$, $r_i=0$ if $s_i=1$, and $r_i\ge1$ otherwise.
Let $u_{k,r,s}=n^{-1}|\{i:(d_i,r_i,s_i)=(k,r,s)\}|$.
On a uniform pairing of the labelled half-edges, run synchronous
progressive activation with these initial states and thresholds,
counting edges with multiplicity, and let $Z_t=|A_t|/n$.
Set $B_{m,r}(x)=\Pr\{\operatorname{Bin}(m,x)\ge r\}$, with
$B_{m,0}=1$ and $B_{m,r}=0$ for $r>m$, and define
\[
 f_u(x)=\sum_{k\ge1,r,s}\frac{k u_{k,r,s}}d B_{k-1,r}(x),
 \qquad g_u(x)=\sum_{k,r,s}u_{k,r,s}B_{k,r}(x),
\]
\[
 x_0=\sum_{k,r}\frac{k u_{k,r,1}}d,\qquad
 y_0=\sum_{k,r}u_{k,r,1},\qquad
 x_{j+1}=f_u(x_j),\quad y_{j+1}=g_u(x_j).
\]
For an integer $t\ge1$, put
\[
 Q_t=\sum_{j=1}^t\Delta^j,\qquad
 b_t=\frac{\Delta Q_t(Q_t+1)}{2M}.
\]
If $M\ge2Q_t+2$, then $|\mathbb E Z_t-y_t|\le b_t$.
For every $h>0$, without this size restriction,
\[
 \Pr\{|Z_t-\mathbb E Z_t|>h\}
 \le2\exp\left(-\frac{nh^2}{128d\Delta^{2t}}\right).
\]
Consequently, if also $b_t\le\varepsilon/2$, then
\[
 \Pr\{|Z_t-y_t|>\varepsilon\}
 \le2\exp\left(-\frac{n\varepsilon^2}
                         {512d\Delta^{2t}}\right).
\]
All these bounds are uniform over the fixed labels.
If instead the labels $(r_i,s_i)$ are iid, independently of the
pairing, with law $\lambda_{r,s}$, the expectation bound holds with
$u_{k,r,s}=p_k\lambda_{r,s}$, where $p_k=n^{-1}|\{i:d_i=k\}|$.
The concentration denominator becomes $(128d+8)\Delta^{2t}$,
and the combined denominator becomes $(512d+32)\Delta^{2t}$.
Here probability averages over pairing and labels.
\end{lemma}

\begin{proof}
Choose a uniform root vertex. Its type has law $u$.
For each nonroot vertex of an independent comparison tree, choose a
uniform half-edge from all $M$ half-edges, with replacement, and use
the type of its incident vertex. These independent types have law
$ku_{k,r,s}/d$. The root has $k$ children and every nonroot vertex
of degree $k$ has $k-1$ children.

On a tree, removing the parent edge does not change a child's
contribution to the first activation of its parent: until that
activation the parent has been inactive at all earlier times.
Induction on time, applying this observation to each child subtree,
therefore identifies the root state with the outward-oriented tree
calculation. With seeds encoded by threshold zero, the states in this
calculation are nondecreasing in time. Independent subtrees then give
exactly the recursions above, and the root activation probability at
time $t$ is $y_t$.

Explore the pairing breadth first from the same root, exposing edges
only from vertices at distance less than $t$. At most $Q_t$ partners
are needed. After $h$ successful tree discoveries at most $h+1$
vertices have been seen. Draw the next comparison half-edge $U$
uniformly from all $M$ half-edges. If it is an eligible partner for
the current half-edge, use it in the pairing; otherwise choose a
uniform eligible partner independently. This rule gives a uniform
partner: every eligible half-edge has probability $1/|E|$, where $E$
is the eligible set. All ineligible half-edges belong to already seen
vertices. Thus, until the first failure, if $U$ belongs to an unseen
vertex, both constructions agree and discover a new vertex. Declare
failure whenever $U$ belongs to a seen vertex. Its conditional
probability is at most $(h+1)\Delta/M$. The size assumption ensures
that partners remain available during this exploration. A union bound
gives failure probability at most
\[
 \sum_{h=0}^{Q_t-1}\frac{(h+1)\Delta}{M}=b_t.
\]
Unexposed edges incident only to depth-$t$ vertices cannot affect the
root state by time $t$. Complete the remaining pairing uniformly.
The coupled root states agree outside the failure event, proving
the expectation bound.

For concentration, generate a uniform pairing by a uniform permutation
of the $M$ half-edges and pair consecutive positions. A transposition
changes at most two pairs and involves at most four endpoint vertices.
A state at time $t$ can change only within distance $t-1$ of those
endpoints in at least one of the two graphs. Each such ball has at
most $\sum_{j=0}^{t-1}\Delta^j\le\Delta^t$ vertices. Hence the
activation count changes by at most $8\Delta^t$ under any
transposition. This estimate also covers loops and repeated edges.

Expose the permutation one position at a time. Conditional completions
for two possible next values are in bijection by transposing those
values in the remaining positions. Their expected activation counts
therefore differ by at most $8\Delta^t$. Each increment of the Doob
martingale for $|A_t|$ is bounded in absolute value by $8\Delta^t$.
Azuma's inequality, with at most $M$ increments, yields
\[
 \Pr\{\big||A_t|-\mathbb E|A_t|\big|>nh\}
 \le2\exp\left(-\frac{n^2h^2}{2M(8\Delta^t)^2}\right),
\]
which is the stated concentration estimate. Taking $h=\varepsilon/2$
and using the expectation bound proves the last assertion.
For iid labels, the same breadth-first coupling assigns each newly
discovered vertex's fresh independent label to its tree occurrence;
at the first repeated vertex, assign that occurrence a fresh independent label
and fill the rest of the comparison tree independently.
Discovery order does not use labels, so this gives mean type law
$p_k\lambda_{r,s}$ and the same bias $b_t$.

For concentration expose the pairing permutation and then the independent
labels. Changing one label affects only its radius-$t$ ball, of size at
most $\sum_{j=0}^t\Delta^j\le2\Delta^t$.
The additional $n$ martingale increments are therefore bounded by
$2\Delta^t$. The squared increment budget is
$(64M+4n)\Delta^{2t}$; Azuma gives denominator
$(128d+8)\Delta^{2t}$. Taking $h=\varepsilon/2$ gives the combined bound.
\end{proof}

Initially use iid Bernoulli$(\alpha)$ seeds, independently of the thresholds
and pairing. The mean tree maps are
\[
 f_\alpha(x)=\alpha+(1-\alpha)F_{\rm P}(x),
\]
\[
 g_\alpha(x)=1-(1-\alpha)\sum_kp_k
 \big[(1-x)^k+(1-\pi)kx(1-x)^{k-1}\big].
\]
Set $x_0=y_0=\alpha$, $x_{t+1}=f_\alpha(x_t)$ and
$y_{t+1}=g_\alpha(x_t)$.

Lemma~\ref{lem:barrier} applies to the mean P response and
$\alpha=(1+\gamma/2)A+o(A)$. The increments
$f_\alpha(x)-x$ are at least $k d^{-4}$ up to $L/d^3$,
$k d^2x^2$ up to $\varepsilon/d$, and $k$ up to $1/2$.
The first interval takes $O(d)$ steps. Crossing each dyadic
$[s,2s]$ takes at most $1+C/(d^2s)$ steps. Their sum is $O(d)$,
since $s$ starts at $L/d^3$ and $\sum_j2^{-j}<\infty$.
The remaining interval takes $O(1)$ steps. Thus a fixed $L_0$
gives $x_{T_0}\ge1/2$ for $T_0=\lceil L_0d\rceil$.

For $q_t=1-x_t\le1/2$, the displayed responses imply, eventually,
\[
 q_{t+1}\le(1+\Delta)q_t^{\delta-2}\le q_t^{\delta/2},
 \qquad 1-g_\alpha(x_t)\le(1+\Delta)q_t^{\delta-1}
                                  \le q_t^{\delta/2}.
\]
The second inequalities use
$\log(1+\Delta)=o(\delta)$, and remain valid at $q_t=0$.
Choose
\[
 r=\max\left\{0,
 \left\lceil\frac{\log((\log n)/(3\log2))}{\log(\delta/2)}
 \right\rceil\right\},\qquad T=T_0+r+1.
\]
Then $1-y_T\le n^{-1/3}$ and
$T\log(\Delta+1)=O(d\log d+\log\log n)=o(\log n)$.

Apply the iid-label part of Lemma~\ref{lem:undirected-finite-time}
at $t=T$ with $\varepsilon=n^{-1/3}$.
Since $Q_T\le2\Delta^T=n^{o(1)}$, eventually
$nd\ge2Q_T+2$ and $b_T=n^{-1+o(1)}\le\varepsilon/2$.
The failure probability is at most
\[
 2\exp\left[-\frac{n^{1/3}}{(512d+32)\Delta^{2T}}\right].
\]
Condition now on exactly $a_0$ seeds and on simplicity.
The seed count $a_0$ is a mode of $\operatorname{Bin}(n,a_0/n)$,
with probability at least $1/(n+1)$; its conditioning gives uniform
seeds of size $a_0$. Pairing and labels are independent, so the threshold
law remains iid and independent of these seeds and the uniform simple
graph. The McKay--Wormald estimate from Section~\ref{sec:v} applies
in the slow range and gives $s_n\ge e^{-C_0d^2}>0$, also proving graphicality.
The resulting total failure is at most
\[
 2(n+1)\exp\left[
 C_0d^2-\frac{n^{1/3}}{(512d+32)\Delta^{2T}}\right]=o(1),
\]
because $d=n^{o(1)}$, $\Delta^{2T}=n^{o(1)}$, and the negative
exponent dominates both $d^2$ and $\log(n+1)$.
On the good event $n-|A_T|\le2n^{2/3}\le3n^{2/3}$.
Permanent activation gives $A_T\subseteq A_\infty$, proving the remainder.
This argument establishes no subcritical terminal upper bound and
does not apply unchanged to the independent-edge model.

\section{A bounded tail above threshold two}\label{app:bounded-tail}
The distributional response equation for arbitrary thresholds is established
in \citet[Section~VI, equations~(23)--(24)]{baxter}. The extension below
concerns quantitative fixed-margin terminal bounds in the growing-degree
regimes of Theorems~\ref{thm:v} and~\ref{thm:er}; the general
threshold model and its response equation are established prior work.

Fix an integer $r_0\ge3$ and $q_*>0$. For each $n$, let the iid thresholds,
independent of the graph and of the uniform seed set of deterministic size
$a$, have probabilities
\begin{equation}\label{tail:law}
 q_{1,n}=\pi=\frac{1-c/d}{d},\qquad q_{2,n}\ge q_*,\qquad
 q_{r,n}\ge0,\quad \sum_{r=1}^{r_0}q_{r,n}=1,
 \quad c>0.
\end{equation}
No limit of the individual probabilities is required. Equivalently, the
tail mass $\eta_n=\Pr(R\ge3)$ is taken from threshold two, so
$q_{2,n}=1-\pi-\eta_n$; the probability of threshold one is unchanged.
Write $T=|A_\infty|$ and $q_2=q_{2,n}$. Seed sizes are
deterministic integers $a_n\in\{0,\ldots,n\}$; each displayed
inequality on a seed-size sequence is required eventually.

\begin{proposition}[Fixed-margin extension with a bounded tail]\label{tail:prop}
Fix $c,C,\kappa>0$ and $\gamma\in(0,1)$ (allow also $C=0$).
For the prescribed model, assume the hypotheses on the degree sequence
in Theorem~\ref{thm:v}: exact mean $d$, even degree sum,
$\max_i|d_i-d|\le C\sqrt d$, $d\to\infty$,
$c+1-v_n\ge\kappa$, where $v_n=\operatorname{Var}(D)/d$, and
$d=o(n^{1/7})$. Set
\[
 s_P=\frac{n(c+1-v_n)^2}{2q_2d^4}.
\]
On the uniform simple graph with these degrees and thresholds
\eqref{tail:law}, each deterministic sequence $a\le(1-\gamma)s_P$
satisfies $T=O(n/d^3)$ with high probability. Each deterministic sequence
$a\ge(1+\gamma)s_P$ satisfies, with high probability,
\begin{equation}\label{tail:P-residual}
 n-T\le C_1 n\bigl[d^{r_0-1}2^{-d+C\sqrt d}+n^{-1/3}\bigr]=o(n).
\end{equation}
For $d$-regular graphs the supercritical conclusion strengthens to $T=n$.

For $G(n,d/n)$, assume $d\to\infty$ and $d^5/n\to0$, and put
\[
 s_E=\frac{nc^2}{2q_2d^4}.
\]
The corresponding subcritical conclusion is $T=O(n/d^3)$, and the
supercritical conclusion is, with high probability,
\begin{equation}\label{tail:E-residual}
 n-T\le C_2 n(1+d^{r_0-1})e^{-d/2+d/n}+n^{2/3}=o(n).
\end{equation}
Constants and probability-error bounds may depend on the displayed fixed
parameters and on $r_0,q_*$, but are uniform over the threshold probability
vectors in \eqref{tail:law}. No simultaneous guarantee over all seed sets,
shrinking critical window, or extension of the finite-degree corrections
in Section~\ref{sec:finite-degree} is asserted.
\end{proposition}

\begin{proof}
We verify the changes to the deterministic margins and every
threshold-specific input in Sections~\ref{sec:v}--\ref{sec:er} and
Appendix~\ref{lup:appendix}; the exploration and concentration arguments themselves
are retained.

\paragraph{Response and margins.}
Use the two responses
\[
 F_P(z)=\frac1d\E\left[D\sum_{r=1}^{r_0}q_{r,n}
       \Pr\{\operatorname{Bin}(D-1,z)\ge r\mid D\}\right],\qquad
 F_E(z)=\sum_{r=1}^{r_0}q_{r,n}\Pr\{\operatorname{Poi}(dz)\ge r\}.
\]
Let $b=1-F'(0)$ and $k=F''(0)$. With $(D)_j$ denoting a falling factorial,
\begin{align*}
 b_P&=1-\pi\E(D)_2/d
     =\frac{c+1-v_n}{d}+\frac{c(v_n-1)}{d^2},&
 k_P&=(q_2-\pi)\E(D)_3/d,\\
 b_E&=c/d,& k_E&=(q_2-\pi)d^2.
\end{align*}
Thus $b\asymp d^{-1}$ and $k\asymp d^2$, uniformly. Differentiating
binomial or Poisson expectations by finite differences of the bounded
function $j\mapsto\Pr(R\le j)$ gives $|F'''(z)|\le C_3d^3$.
Consequently, for a sufficiently small fixed $\varepsilon>0$,
\begin{equation}\label{tail:taylor}
 F(z)=(1-b)z+\tfrac12kz^2+O(d^3z^3),\qquad
 F''(z)\ge c_1d^2\quad(0\le z\le\varepsilon/d).
\end{equation}
The exact local response ratio $H(z)=(z-F(z))/(1-F(z))$ has a unique
maximum in this interval: the numerator of $H'$ is
$1-F-(1-z)F'$, whose derivative is $-(1-z)F''<0$.
Its maximizer and value satisfy
\[
 z_*=\frac bk(1+O(d^{-2})),\qquad
 A=\frac{b^2}{2k}(1+O(d^{-2})).
\]
These follow from \eqref{tail:taylor} at $z\asymp d^{-3}$; the
denominator $1-F(z)=1+O(d^{-3})$ there. In particular,
\[
 A_P=\frac{(c+1-v_n)^2}{2q_2d^4}(1+O(d^{-1})),\qquad
 A_E=\frac{c^2}{2q_2d^4}(1+O(d^{-1})).
\]

For completeness, the higher thresholds cause no second obstruction
before $z=1/2$. On $[\varepsilon/d,L_0/d]$, the threshold-two mass
alone gives $F(z)\ge c_\varepsilon>0$, uniformly for sufficiently
large $n$; hence $F(z)-z\ge c_\varepsilon/2$. Choose the fixed $L_0$
large enough that both binomial and Poisson probabilities of at least
$r_0$ marks at $L_0/d$ exceed $3/4$, uniformly in the allowed degrees.
Then $F(z)-z\ge1/4$ on $[L_0/d,1/2]$.
Together with \eqref{tail:taylor}, these estimates give the three
fixed-margin bounds for $h_\alpha=\alpha+(1-\alpha)F-z$:
\begin{equation}\label{tail:margins}
 h_{\alpha_+}(z)\ge c_2
 \begin{cases}
 d^{-4},&0\le z\le L/d^3,\\
 d^2z^2,&L/d^3\le z\le\varepsilon/d,\\
 1,&\varepsilon/d\le z\le1/2,
 \end{cases}
 \qquad h_{\alpha_-}(b/k)\le-c_2d^{-4}.
\end{equation}
Here $\alpha_\pm$ are fixed relative offsets above/below $A$, and $L$
is a sufficiently large fixed constant. To see the middle bound, the
linear loss $bz$ is absorbed by $kz^2/2$ by increasing $L$, and the
cubic remainder is absorbed by decreasing $\varepsilon$. On the early
interval the quadratic minimum equals $-b^2/(2k)$ and its remainder
is $O(d^{-6})$. This also proves that the above local maximum is the
maximum on $[0,1/2]$. Integer seed rounding is negligible since
$n/d^4\to\infty$ in both regimes. More explicitly, for the applicable
leading scale $s=s_P$ or $s_E$, choose deterministic reference sizes
\[
 a_- = \lfloor(1-\gamma/2)s\rfloor,\qquad
 a_+ = \lceil(1+\gamma/2)s\rceil.
\]
Because $s/(nA)=1+O(d^{-1})$ uniformly over the allowed threshold
laws and $s\asymp n/d^4\to\infty$, these references retain relative
margins at least $\gamma/4$ below and above $nA$ eventually.
Every permitted subcritical size is at most $a_-$ and every permitted
supercritical size is at least $a_+$. Coupling uniform seeds by one
uniform permutation, as in the
\hyperref[par:seed-transfer]{integer-seed argument}, transfers
the reference conclusions to each permitted deterministic seed-size sequence.

\paragraph{Prescribed exploration and local concentration.}
The clock exploration of Section~\ref{sec:v} is valid for every
positive integer threshold. Before activation no end at that vertex has
been manually removed. Thus, for $z=1-e^{-u}$, its inactive-end process
and exact mean remain
\[
 H_0(u)=\sum_{i\notin A_0}(d_i-X_i(u))\mathbf{1}_{\{X_i(u)<R_i\}},\qquad
 \E H_0(u)=(1-\alpha)nd(1-z)(1-F_P(z)).
\]
The active-queue drift is $nd(1-z)h_\alpha(z)$ and the stopping
convention remains $Q(\tau)=-1$.

Here is the only threshold-specific variance replacement needed in
Appendix~\ref{lup:appendix}. For $X\sim\operatorname{Bin}(j,z)$ independent of $R$, set
$Y=(j-X)\mathbf{1}_{\{X<R\}}$ and $L_R=j-Y$.
Couple $R'=1$ if $R=1$ and $R'=2$ otherwise. Pointwise,
$0\le L_R\le L_{R'}$, and therefore
\begin{equation}\label{tail:loss}
 \operatorname{Var}(Y)\le\E L_R^2\le\E L_{R'}^2
 \le\pi j^3z+jz+\tfrac12j^4z^2.
\end{equation}
This is exactly the bound used in \eqref{lup:variance}.
The fixed-size seed-mean inequality~\eqref{lup:seed} uses only $0\le\E Y_i\le\Delta$.
The stopped queue martingale estimate~\eqref{lup:W} depends only on pairing rates, not
threshold values. The interpolation step~\eqref{lup:interval} still holds: $Y_i(u)$ is
nonincreasing, its event rate is at most $d_i$, and each loss is at most
$d_i$, so $|(\E Y_i)'|\le d_i^2$.
All three scale budgets therefore apply with \eqref{tail:margins}, giving
the same form of unconditioned failure bound
\[
 C_4(1+d^2)(1+\log d)\exp(-c_3n/d^5).
\]
The negative test forces stopping before $z=b/k$, and the positive
margins prevent stopping up to $u=\log2$.

For the subcritical reference size, put $z_*=b/k$ and $u_*=-\log(1-z_*)$.
On $\{\tau<u_*\}$, every eventually active nonseed vertex has reached its
threshold by $u_*$. Hence
\[
 T\le a_-+\sum_i\mathbf{1}_{\{X_i(u_*)\ge R_i\}}.
\]
The sum has unconditional mean at most
$\sum_i[\pi d_i z_*+\binom{d_i}{2}z_*^2]=O(n/d^3)$, since replacing
$R$ by $R'$ only increases it. The Bernoulli concentration argument of
Section~\ref{sec:v}, together with $a_-=O(n/d^4)$, gives the stated
subcritical bound.

For the supercritical reference size, put $u=\log 2$. On the passage
event $\{\tau>u\}$, each vertex still inactive at $u$ satisfies
$X_i(u)<R_i\le r_0$. The count $\sum_i\mathbf{1}_{\{X_i(u)<r_0\}}$
therefore bounds the number of such vertices on this event and has
unconditional mean at most
\[
 \sum_i 2^{-d_i}\sum_{j=0}^{r_0-1}\binom{d_i}{j}
 \le C_5n d^{r_0-1}2^{-d+C\sqrt d}.
\]
Hoeffding's inequality adds $n^{2/3}$ with failure at most
$\exp(-2n^{1/3})$. As in the proof in Section~\ref{sec:v}, these unconditional events
are intersected with passage, without conditioning their laws on survival.
The graph and its simplicity probability are unchanged. Dividing failure
probabilities by the lower bound from Section~\ref{sec:v} $\exp(-C_6d^2)$ is harmless:
$n/d^5\gg d^2$ and $n^{1/3}\gg d^2$ when $d=o(n^{1/7})$.
Conditioning on simplicity preserves the independent seed and threshold
laws. This proves \eqref{tail:P-residual}.

\paragraph{Independent edges.}
The iid-mark construction and the exact stopping identity~\eqref{eq:lue-stop} also hold for the new thresholds. With
$Q_n(t)=\Pr\{\operatorname{Bin}(t,d/n)\ge R\}$, the nonseed count is
$\operatorname{Bin}(n-a,Q_n(t))$ at each deterministic time. The
distribution-free maximal inequality of Lemma~\ref{lem:er-maximal} applies unchanged.
For $z=t/n$ the same two estimates remain valid:
\[
 Q_n(t)\le z+\tfrac12d^2z^2,\qquad
 |Q_n(t)-F_E(z)|\le d^2z/n.
\]
The first follows by the $R'$ domination; the second couples each
Bernoulli mark to a Poisson mark using the same $R$. With
\eqref{tail:margins}, the early, dyadic middle and late maximal budgets
in \eqref{eq:er-mean-margins}--\eqref{eq:er-passage} give failure $O(d^5/n)$.
For example the early CDF is $O(d^{-3})$ and its drift budget is
of order $n/d^4$; the middle failures sum as
$\sum_jO((nd^4s_j^3)^{-1}+(nd^2s_j^2)^{-1})=O(d^5/n)$.
The Poisson-coupling and time-rounding errors remain negligible at
these scales. Thus the negative test gives $T=O(n/d^3)$, and the
positive margins give $T>\lfloor n/2\rfloor$.
For $m=\lfloor n/2\rfloor$, $p=d/n\le1/2$,
\[
 1-Q_n(m)\le(1-p)^m
       \sum_{j=0}^{r_0-1}\binom mj\left(\frac p{1-p}\right)^j
 \le C_7(1+d^{r_0-1})e^{-d/2+d/n}.
\]
Intersecting the unconditional binomial concentration event at $m$
with passage yields \eqref{tail:E-residual}. Uniform permutation
coupling transfers both reference outcomes to every permitted
deterministic seed size, separately as in the \hyperref[par:seed-transfer]{integer-seed argument}.

\paragraph{Regular completion.}
For a closed inactive set $U$, every $v\in U$ has at most $r_0-1$
neighbors outside $U$. If $0<|U|\le n/2$, the spectral cut bound used
in the proof of Lemma~\ref{lem:regular-completion} gives
\[
 (d-\lambda_*)|U|/2\le e(U,U^c)\le(r_0-1)|U|.
\]
The spectral input of \citet[Theorem~1.1]{cook2018}, as used in the proof of
Corollary~\ref{cor:regular}, gives $\lambda_*=O(\sqrt d)$ with high
probability, hence $d-\lambda_*>2(r_0-1)$ eventually. The preceding
almost-complete result reaches half the graph with high probability,
so a nonempty terminal $U$ is impossible. This deterministic implication
does not require independence of the spectral and activation events.
\end{proof}

\paragraph{Comparison with the two-threshold model.}
The linear deficit is unchanged, whereas the curvature is
$(q_2-\pi)\E(D)_3/d$, or $(q_2-\pi)d^2$ in the Poisson response.
The leading seed scales are multiplied by $1/q_2$ relative to the
two-threshold coefficients in Theorems~\ref{thm:v} and~\ref{thm:er}. The terminal residual has a
polynomial factor of degree $r_0-1$ instead of one; the regular
completion condition becomes $d-\lambda_*>2(r_0-1)$.

\section*{Author disclosure}
The author is employed by Yandex. This research was conducted independently
in the author's personal time and received no dedicated funding from Yandex.
Some computational work used Yandex equipment and computing infrastructure.
The views expressed are solely those of the author and do not necessarily
represent those of Yandex. The author declares no other competing interests.

\section*{AI assistance}
The author used OpenResearch with OpenAI models to organize the study, review the literature, perform calculations, develop and review proofs, and edit and review the manuscript.

\end{document}